\newif\ifpoly
\polytrue 

\documentclass[reqno,12pt,a4paper]{article}
\usepackage{ifthen}
\usepackage{a4wide}
\usepackage{rotating}
\usepackage{amsmath}
\usepackage{amsfonts}
\usepackage{amssymb}
\usepackage{amscd}
\usepackage{amsthm}
\usepackage{stmaryrd}

\SetSymbolFont{stmry}{bold}{U}{stmry}{m}{n}

\usepackage[all]{xy}
\SelectTips{cm}{10}
\usepackage{graphics}
\usepackage{graphicx}
\usepackage{fancyhdr}
\usepackage{ushort}
\usepackage{enumitem}
\usepackage{color}
\usepackage{mathtools}
\usepackage{url}
\usepackage{float}

\numberwithin{equation}{section}

\theoremstyle{plain}
\newtheorem{theorem}[equation]{Theorem}
\newtheorem{lemma}[equation]{Lemma}
\newtheorem{corollary}[equation]{Corollary}
\newtheorem{proposition}[equation]{Proposition}
\newtheorem{addendum}[equation]{Addendum}

\theoremstyle{definition}
\newtheorem{definition}[equation]{Definition}
\newtheorem{convention}[equation]{Convention}
\newtheorem{nexample}[equation]{Example}
\newtheorem*{example}{Example}

\newcommand{\heading}[1]{\setcounter{subsection}{\value{equation}}\stepcounter{equation}\subsection{#1}}

\theoremstyle{remark}
\newtheorem*{remark}{Remark}

\newcommand{\comma}{{},\,}
\newcommand{\col}{\colon}

\newcommand{\bbZ}{\mathbb Z}
\newcommand\Z{\bbZ}
\newcommand{\bbR}{\mathbb R}
\newcommand{\mcA}{\mathcal A}
\newcommand{\mcB}{\mathcal B}
\newcommand{\mcC}{\mathcal C}
\newcommand{\mcD}{\mathcal D}

\newcommand{\mcJ}{\mathcal J}
\newcommand{\mcS}{\mathcal S}

\newcommand{\sfP}{\mathsfsl{P}}
\newcommand{\sfQ}{\mathsfsl{Q}}

\DeclareMathAlphabet{\mathsfsl}{OT1}{cmss}{m}{sl}

\newcommand{\Map}{\mathsfsl{Map}}
\newcommand{\Pair}{{\mathsfsl{Pair}}}
\newcommand{\SSet}{\mathsfsl{SSet}}
\newcommand{\Alg}{\mathsfsl{Alg}}

\newcommand{\MPS}[1]{\mathsfsl{MPS}_{#1}}
\newcommand{\EMPS}[1]{\mathsfsl{EMPS}_{#1}}
\newcommand{\PMPS}[2][]{\ifthenelse{\equal{#1}{}}{\mathsfsl{PMPS}_{#2}}{\mathsfsl{PMPS}_{#2,#1}}}
\newcommand{\HMPS}[2][]{\ifthenelse{\equal{#1}{}}{\mathsfsl{HMPS}_{#2}}{\mathsfsl{HMPS}_{#2,#1}}}

\newcommand{\xra}[1]{\xrightarrow{#1}}

\newcommand{\ra}{\rightarrow}
\newcommand{\la}{\leftarrow}
\newcommand{\xlra}[1]{\xrightarrow{\ #1\ }}
\newcommand{\xlla}[1]{\xleftarrow{\ #1\ }}
\newcommand{\lra}{\longrightarrow}

\newcommand{\Ra}{\xRightarrow{\ \ }}
\newcommand{\La}{\xLeftarrow{\ \ }}
\newcommand{\LRa}{\xLeftrightarrow{\ \ \ }}

\newdir{c}{{}*!/-5pt/@^{(}}
\newdir{d}{{}*!/-5pt/@_{(}}
\newdir{ >}{{}*!/-5pt/@{>}}
\newdir{s}{{}*!/+10pt/@{}}
\newdir{|>}{{}*!/2pt/@{|}*@{>}}

\newcommand{\pbsize}{15pt}
\newcommand{\pboffset}{.5}
\newcommand{\xycorner}[3]{\save #2="a";#1;"a"**{}?(\pboffset);"a"**\dir{-};#3;"a"**{}?(\pboffset);"a"**\dir{-}\restore}
\newcommand{\pb}{\xycorner{[]+<\pbsize,0pt>}{[]+<\pbsize,-\pbsize>}{[]+<0pt,-\pbsize>}}

\newcommand{\xymatrixc}[1]{\xy *!C\xybox{\xymatrix{#1}}\endxy}

\newcommand{\cof}[1][]{\mathbin{\:\!\!\xymatrix@1@C=15pt{{}\ar@{ >->}[r]^{#1} & {}}}}
\newcommand{\fib}[1][]{\mathbin{\:\!\!\xymatrix@1@C=15pt{{}\ar@{->>}[r]^{#1} & {}}}}
\newcommand{\family}[1][]{\mathbin{\:\!\!\xymatrix@1@C=15pt{{}\ar@{~>}[r]^{#1} & {}}}}
\newcommand{\pto}[1][]{\mathbin{\:\!\!\xymatrix@1@C=20pt{{}\ar@{-->}[r]^{#1} & {}}}} 

\newcommand{\coker}{\mathop\mathrm{coker}\nolimits}
\newcommand{\cone}{\mathop\mathrm{cone}\nolimits}
\newcommand{\cyl}{\mathop\mathrm{cyl}\nolimits}
\newcommand{\defeq}{\stackrel{\mathrm{def}}{=}}
\newcommand{\ef}{\mathrm{ef}}
\newcommand{\ev}{\mathop\mathrm{ev}\nolimits}

\newcommand{\id}{\mathop\mathrm{id}\nolimits}

\newcommand{\im}{\mathop\mathrm{im}\nolimits}
\newcommand{\map}{\mathop\mathrm{map}\nolimits}
\newcommand{\op}{\mathrm{op}}
\newcommand{\pr}{\mathop\mathrm{pr}\nolimits}
\newcommand{\size}{\mathop\mathrm{size}\nolimits}
\newcommand{\sk}{\mathop\mathrm{sk}\nolimits}

\newlength{\hlp}

\newcommand{\leftbox}[2]{\settowidth{\hlp}{$#1$}\makebox[\hlp][l]{${#1}{#2}$}}
\newcommand{\rightbox}[2]{\settowidth{\hlp}{$#2$}\makebox[\hlp][r]{${#1}{#2}$}}

\entrymodifiers={+!!<0pt,\the\fontdimen22\textfont2>}

\newcommand{\Hopf}{H-}
\newcommand{\bdry}{d}
\newcommand{\ZG}{{\bbZ G}}

\newcommand{\hnabla}{\mathop{\widehat\nabla}}
\newcommand{\hDelta}{\mathop{\widehat\Delta}}
\newcommand{\htimes}{\mathbin{\widehat\times}}
\newcommand{\hvee}{\mathbin{\widehat\vee}}

\newcommand{\add}{\mathop\mathrm{add}\nolimits}
\newcommand{\inv}{\mathop\mathrm{inv}\nolimits}
\newcommand{\hadd}{\mathop\mathrm{add'}\nolimits}
\newcommand{\hplus}{\mathbin{+'}}

\newcommand{\theconn}{{d}}
\newcommand{\thedim}{{n}}
\newcommand{\theDim}{{n_0}}
\newcommand{\theotherdim}{{m}}
\newcommand{\thedimm}{{i}}
\newcommand{\thedimmm}{{j}}
\newcommand{\stdsimp}[1]{\Delta^{#1}}
\newcommand{\horn}[2]{%
\mbox{$\xy
<0pt,-\the\fontdimen22\textfont2>;p+<.1em,0em>:
{\ar@{-}(0,0.1);(3,7)},
{\ar@{-}(3,7);(6,0.1)},
{\ar@{-}(3.2,7);(6.2,0.1)},
{\ar@{-}(3.4,7);(6.4,0.1)}
\endxy\;\!{}^{#1}_{#2}$}}
\newcommand{\vertex}[1]{#1}

\newcommand{\Pnew}{{P_\thedim}}
\newcommand{\Pold}{{P_{\thedim-1}}}
\newcommand{\Polder}{{P_{\thedimm}}}
\newcommand{\Pm}{{P_\theotherdim}}
\newcommand{\tPnew}{{\widetilde P_\thedim}}
\newcommand{\tPold}{{\widetilde P_{\thedim-1}}}
\newcommand{\tB}{{\widetilde B}}

\newcommand{\pin}{{\pi_\thedim}}

\newcommand{\onew}{{o_\thedim}}
\newcommand{\oold}{{o_{\thedim-1}}}

\newcommand{\ommo}{{o_{\theotherdim-1}}}
\newcommand{\om}{{o_\theotherdim}}

\newcommand{\Kn}{{K_{\thedim+1}}}
\newcommand{\En}{{E_\thedim}}
\newcommand{\Ln}{{L_\thedim}}
\newcommand{\Li}{{L_\thedimm}}

\newcommand{\kn}{{k_\thedim}}

\newcommand{\knp}{{k_\thedim'}}
\newcommand{\kip}{{k_\thedimm'}}
\newcommand{\knst}{{k_{\thedim*}}}

\newcommand{\knstinv}{{k_{\thedim*}^{-1}}}

\newcommand{\kmst}{{k_{\theotherdim*}}}

\newcommand{\konepef}{{\kappa_1^\ef}}
\newcommand{\knpef}{{\kappa_\thedim^\ef}}
\newcommand{\kipef}{{\kappa_\thedimm^\ef}}

\newcommand{\pn}{{p_\thedim}}
\newcommand{\pnst}{{p_{\thedim*}}}
\newcommand{\qnp}{{q_\thedim'}}

\newcommand{\qn}{{q_\thedim}}

\newcommand{\jn}{{j_\thedim}}
\newcommand{\jnst}{{j_{\thedim*}}}

\newcommand{\kno}{{\kn\oold}}

\newcommand{\qno}{{\qn\onew}}

\newcommand{\fn}{{f_\thedim}}

\newcommand{\varphin}{{\varphi_\thedim}}
\newcommand{\varphinp}{{\varphi_\thedim'}}
\newcommand{\varphinst}{{\varphi_{\thedim*}}}
\newcommand{\varphinpst}{{\varphi_{\thedim*}'}}

\newcommand{\tfn}{{\widetilde f_\thedim}}

\newcommand{\psin}{{\psi_\thedim}}
\newcommand{\psinst}{{\psi_{\thedim*}}}
\newcommand{\psiold}{{\psi_{\thedim-1}}}
\newcommand{\tpsin}{{\widetilde\psi_\thedim}}
\newcommand{\tonew}{{\widetilde o_\thedim}}

\newcommand{\connn}{{\partial_\thedim}}
\newcommand{\conni}{{\partial_\thedimm}}

\usepackage{titlesec}

\titleformat{\section}[block]
{\normalfont\Large\filcenter\bfseries}{\thesection.}{.33em}{}
\titlespacing*{\section}
{0pt}{3.5ex plus 1ex minus .2ex}{2.3ex plus .2ex}

\titleformat{\subsection}[runin]
{\normalfont\normalsize\bfseries}{\thesubsection.}{.33em}{}[.]
\titlespacing*{\subsection}
{0pt}{3.25ex plus 1ex minus .2ex}{.5em}

\def\immediateFigure#1{%
\smallskip\begin{center}#1\end{center}\smallskip }
\newcommand{\immfig}[1]  
{\immediateFigure{\mbox{\includegraphics{#1}}}}

\usepackage{ifthen}
\usepackage[strict]{changepage}
\usepackage{mparhack}

\title
{Algorithmic solvability of the lifting-extension problem\thanks{%
The research of M.\ \v{C}.\ was supported by the project CZ.1.07/2.3.00/20.0003 of the Operational Programme Education for Competitiveness of the Ministry of Education, Youth and Sports of the Czech Republic. The research by M.\ K.\  was supported by the Center of Excellence -- Inst.\ for Theor.\ Comput.\ Sci., Prague (project P202/12/G061 of GA~\v{C}R) and by the Project LL1201 ERCCZ CORES. The research of L.\ V.\ was supported by the Center of Excellence -- Eduard \v{C}ech Institute (project P201/12/G028 of GA~\v{C}R).
\vskip .5ex \noindent
\emph{2010 Mathematics Subject Classification}. Primary 55Q05; Secondary 55S91.
\vskip .5ex \noindent
\emph{Key words and phrases}. Homotopy classes, equivariant, fibrewise, lifting-extension problem, algorithmic computation, embeddability, Moore--Postnikov tower.
}}

\author
{Martin \v{C}adek \and Marek Kr\v{c}\'al \and Luk\'a\v{s} Vok\v{r}\'{\i}nek}

\begin{document}

\maketitle

\begin{abstract}
Let $X$ and $Y$ be finite simplicial sets (e.g.\ finite simplicial complexes), both equipped with a free simplicial action of a finite group $G$. Assuming that $Y$ is $d$-connected and $\dim X\le 2d$, for some $d\ge 1$, we provide an algorithm that computes the set of all equivariant homotopy classes of equivariant continuous maps $|X|\to|Y|$; the existence of such a map can be decided even for $\dim X\leq 2d+1$.
\ifpoly
For fixed $G$ and $\theconn$, the algorithm runs in polynomial time.
\fi
This yields the first algorithm for deciding topological embeddability of a $k$-dimensional finite simplicial complex into $\bbR^n$ under the condition $k\leq\frac 23 n-1$.

More generally, we present an algorithm that, given a lifting-extension problem satisfying an appropriate stability assumption, computes the set of all homotopy classes of solutions. This result is new even in the non-equivariant situation.
\end{abstract}

\section{Introduction}

Our original goal for this paper was to design an algorithm that decides existence of an equivariant map between given spaces under a certain ``stability'' assumption. To explain our solution however, it is more natural to deal with a more general lifting-extension problem. At the same time, lifting-extension problems play a fundamental role in algebraic topology since many problems can be expressed as their instances. We start by explaining our original problem and its concrete applications and then proceed to the main object of our study in this paper -- the lifting-extension problem.

\subsection*{Equivariant maps}
Consider the following
algorithmic problem: given a finite group $G$ and 
two free $G$-spaces $X$ and $Y$, decide the existence of 
an equivariant map $f\col X\to Y$.

In the particular case $G=\Z/2$ and $Y=S^{n-1}$ equipped with the antipodal $\Z/2$-action, this problem has various applications in geometry and combinatorics.

Concretely, it is well-known that if a simplicial complex $K$ embeds into $\bbR^{\thedim}$ then there exists a $\Z/2$-equivariant map $(K \times K) \smallsetminus \Delta_K \to S^{n-1}$; the converse holds in the so-called \emph{metastable range} $\dim K \leq \tfrac{2}{3} \thedim-1$ by \cite{Weber:PlongementsPolyedresDomaineMetastable-1967}. Algorithmic aspects of the problem of embeddability of $K$ into $\bbR^{\thedim}$ were studied in \cite{MatousekTancerWagner:HardnessEmbeddings-2011} and, with the exception of low dimensions, the meta-stable range was the only remaining case left open.
\ifpoly
Theorem~\ref{t:emb-metast} below shows that, for fixed $n$, it is solvable in polynomial time.
\fi

Equivariant maps also provide interesting applications of topology to combinatorics. For example, the celebrated result of Lov\'asz on Kneser's conjecture states that for a graph $G$, the absence of a $\Z/2$-equivariant map $B(G) \to S^{n-1}$ imposes a lower bound $\chi(G) \geq n+2$ on the chromatic number of $G$, where $B(G)$ is a certain simplicial complex constructed from $G$, see \cite{Mat-top}.

Building on the work of Brown \cite{Brown}, which is not applicable for $Y=S^{\thedim-1}$, we investigated in papers \cite{CKMSVW11,polypost} the
simpler, non-equivariant situation, where $X$ and $Y$ were topological
spaces and we were interested in $[X,Y]$, the set of all homotopy
classes of continuous maps $X\to Y$. Employing methods of \emph{effective homology}
developed by Sergeraert et al.\ (see e.g.\ \cite{SergerGenova}),
we showed that for any fixed $d\ge 1$, $[X,Y]$ is polynomial-time computable
if $Y$ is $d$-connected and $\dim X\le 2d$.\footnote{%
	An extension of \cite{CKMSVW11} to the case of a simply connected $Y$ whose non-stable homotopy groups, i.e.\ the groups $\pin(Y)$ for $\thedim > 2 \theconn$, are finite (e.g.\ an odd-dimnsional sphere) that works for $X$ of arbitrary dimension can be found in~\cite{odd-spheres}.
} In contrast, \cite{ext-hard} shows that the problem of computing $[X,Y]$ is \#P-hard when the dimension restriction on $X$ is dropped. More strikingly, a related problem of the existence of a continuous extension of a given map $A \to Y$, defined on a subspace $A$ of $X$, is \emph{undecidable} as soon as $\dim X \geq 2d+2$.

Here we obtain an extension of the above computability result for free $G$-spaces and equivariant maps. The input $G$-spaces $X$ and $Y$ can be given as finite \emph{simplicial sets} (generalizations of finite simplicial complexes, see \cite{Friedm08}), and the free action of $G$ is assumed simplicial. The simplicial sets and the $G$-actions on them are described by a finite table.

\begin{theorem}\label{theorem:equivariant}
Let $G$ be a finite group. There is an algorithm that, given finite simplicial sets $X$ and $Y$ with free simplicial actions of $G$, such that $Y$ is $d$-connected, $d\geq 1$, and $\dim X\leq 2d+1$, decides the existence of a continuous equivariant map $X\to Y$.

If such a map  exists and $\dim X\leq 2d$, then the set $[X,Y]$ of all equivariant homotopy classes of equivariant continuous maps can be equipped with the structure of a finitely generated abelian group, and the algorithm outputs the isomorphism type of this group.
\ifpoly

For fixed $G$ and $\theconn$, this algorithm runs in polynomial time.
\fi
\end{theorem}

The isomorphism type is output as an abstract abelian group given by a (finite) number of generators and relations. Furthermore, there is an algorithm that, given an equivariant simplicial map $\ell\col X\to Y$, computes the element of this group that $\ell$ represents. In the opposite direction, although every homotopy class can be represented by a simplicial map $X' \to Y$ for some \emph{subdivision} $X'$ of $X$, we do not know of effective means of producing such representatives.\footnote{%
	It is possible, for a given homotopy class $z \in [X, Y]$, to go through all subdivisions $X'$ and all possible simplicial maps $X' \to Y$ and test if they represent $z$. However, such a procedure does not seem to be very effective.
}

As a consequence, we also have an algorithm that, given two equivariant simplicial maps $X\to Y$,
tests whether they are equivariantly homotopic under the above dimension restrictions on $X$.
Building on the methods of the present paper, \cite{Filakovsky:suspension} removes the
dimension restriction for the latter question: it provides a homotopy-testing algorithm
assuming only that $Y$ is simply connected.

A work in progress has a goal to extend the results of the present paper to non-free $G$-actions; for this extension, it seems necessary to work with diagrams of fixed points of various subgroups $H \leq G$ and maps between them, while free actions allow to work with a single space (namely, the fixed points of the trivial subgroup).

\subsection*{Lifting-extension problem}
We obtain Theorem~\ref{theorem:equivariant} by an inductive approach that works more generally and more naturally in the setting of the \emph{(equivariant) lifting-extension problem}, summarized in the following diagram:
\begin{equation}\label{eq:basic_square}
\xy *!C\xybox{\xymatrix@C=40pt{
A \ar[r]^-{f} \ar@{ >->}[d]_-\iota & Y \ar@{->>}[d]^-\psi \\
X  \ar[r]_-{g} \ar@{-->}[ru]^-{\ell
} & B
}}\endxy
\end{equation}
The input objects for this problem are the solid part of the diagram and we require that:
\begin{enumerate}[labelindent=.5em,leftmargin=*,label=$\bullet$,itemsep=0pt,parsep=0pt,topsep=5pt]
\item $A$, $X$, $Y$, $B$ are free $G$-spaces;
\item $f\col A\to Y$ and $g\col X\to B$ are equivariant maps;
\item $\iota\col A\cof X$ is an equivariant \emph{cofibration} (simplicially: an inclusion);
\item $\psi\col  Y\fib B$ is an equivariant \emph{fibration} (simplicially: a Kan fibration, see \cite{May:SimplicialObjects-1992}); and
\item the square commutes (i.e.\ $g\iota=\psi f$).
\end{enumerate}

The lifting-extension problem asks whether there exists a \emph{diagonal} in the square, i.e.\ an equivariant map
$\ell\col X\to Y$, marked by the dashed arrow, that makes both
triangles commute. We call such an $\ell$ a
\emph{solution} of the lifting-extension
problem~\eqref{eq:basic_square}.

Moreover, if such an $\ell$ exists, we would like to compute the set $[X,Y]^A_B$ of all solutions up to equivariant \emph{fibrewise} homotopy \emph{relative to $A$}.\footnote{%
	A homotopy $h\col [0,1]\times X\to Y$ is \emph{fibrewise} if $\psi(h(t,x))=g(x)$ for all $t\in[0,1]$ and $x\in X$. It is \emph{relative to $A$} if, for $a\in A$, $h(t,a)$ is independent of $t$, i.e.\ $h(t,a)=f(a)$ for all $t\in [0,1]$ and $a\in A$.
}
More concretely, in the cases covered by our algorithmic results, we will be able to equip $[X,Y]^A_B$ with a structure of an abelian group, and the algorithm computes the isomorphism type of this group. To be more precise, this structure is only canonical up to a choice of zero, with various choices differing by translations, so that $[X,Y]_B^A$ really has an ``affine'' nature (in very much the same way as an affine space is naturally a vector space up to a choice of its origin). For an abstract point of view, see~\cite{heaps1}.

\subsection*{Generalized lifting-extension problem}
Spaces appearing in a fibration $\psi \col Y\fib B$ must typically be represented
by infinite simplicial sets%
\footnote{%
	If $\psi$ is a Kan fibration between finite simply connected simplicial sets then its fibre is a finite Kan complex and it is easy to see that it then must be discrete. Consequently, $\psi$ is a covering map between simply connected spaces and thus an isomorphism.}%
, and their representation
as inputs to an algorithm can be problematic. For this reason,
we will consider a \emph{generalized lifting-extension problem},
where, compared to the above, $\psi\col Y\to B$ can be an arbitrary equivariant map,
\emph{not} necessarily a fibration.

In this case, it makes no sense from the homotopy point of view to define a solution as a map $X\to Y$ making both triangles commutative. A homotopically correct definition of a solution is as a pair $(\ell,h)$, where $\ell\col  X\to Y$ is a map for which the upper triangle commutes strictly and the lower one commutes up to the specified homotopy $h\col [0,1]\times X\to B$ relative to $A$. We will not pursue this approach any further (in particular, we will not define the right notion of homotopy of such pairs) and choose an equivalent, technically less demanding alternative, which consists in replacing the map $\psi$ by a homotopy equivalent fibration.

To this end, we factor $\psi\col  Y\to B$ as
a weak homotopy equivalence $j\col  Y\xra\sim Y'$ followed by
a fibration $\psi'\col Y'\fib B$ (in the simplicial setup, see Lemma~\ref{l:fibrant_replacement}).
We define a \emph{solution} of the
considered generalized lifting-extension problem to be a solution
$\ell'\col X\to Y'$ of the lifting-extension problem
\[\xymatrixc{
A \ar[r]^-{f} \ar@{ >->}[d]_-\iota & Y \ar[r]^-j & Y' \ar@{->>}[d]^-{\psi'} \\
X  \ar[rr]_-{g} \ar@{-->}[rru]^-{\ell'} & & B
}\]
If $\psi$ was a fibration to begin with, we naturally take $Y=Y'$ and
$j=\id$, and then the two notions of a solution coincide.
With some abuse of notation, we write $[X,Y]_B^A$ for the set
$[X,Y']_B^A$ of all homotopy classes of solutions of the above lifting-extension problem.
Clearly, for every diagonal $\ell \colon X \to Y$ (i.e.\ a map $\ell$ satisfying $f = \ell \iota$ and $g = \psi \ell$), the composition $\ell' = j \ell$ is a solution and, in this way, $\ell$ represents a homotopy class in $[X, Y]^A_B$. On the other hand, not every homotopy class is represented by a diagonal $\ell \colon X \to Y$.

We remark that $Y'$ is used merely as a theoretical tool -- for actual computations, we use a different approximation of $Y$,
namely a suitable finite stage of a Moore--Postnikov tower
for $\psi\col Y\to B$; see Section~\ref{s:main_proofs}. Moreover, $Y'$
is not determined uniquely, and thus neither are the solutions of the generalized lifting-extension problem.
However, rather standard considerations show that the \emph{existence} of a solution and the \emph{isomorphism type}
of $[X,Y']_B^A$ as an abelian group are independent of the choice of $Y'$.

\subsection*{Examples of lifting-extension problems}
In order to understand
the meaning of the (generalized) lifting-extension problem,
it is instructive to consider some special
cases.

\begin{enumerate}[labelindent=.5em,leftmargin=*,label=(\roman*)]
\item (Classification of extensions.)
First, consider $G=\{e\}$ trivial (thus, the equivariance conditions
are vacuous) and $B$ a point (which makes the lower triangle in the
lifting-extension problem superfluous).
Then we have an \emph{extension problem},
asking for the existence of a map $\ell\col X\to Y$ extending a given $f\col A\to Y$. We recall that this problem is undecidable when $\dim X$ is not bounded, according to~\cite{ext-hard}. Moreover, $[X,Y]^A$ is the set of appropriate homotopy classes of such extensions.%
\footnote{
	The problem of computing homotopy classes of solutions (under our usual condition on the dimension of $X$) was considered in \cite{polypost}, but with a different equivalence relation on the set of all extensions: \cite{polypost} dealt with the (slightly unnatural) \emph{coarse classification}, where two extensions $\ell_0$ and $\ell_1$ are considered equivalent if they are homotopic as maps $X\to Y$, whereas here we deal with the \emph{fine classification}, where the equivalence of $\ell_0$ and $\ell_1$ means that they are homotopic relative to~$A$.
}

\item (Equivariant maps.) Consider $G$ finite, $A=\emptyset$, and $B=EG$,
a contractible free $G$-space (it is unique up to equivariant homotopy equivalence).
For every free $G$-space $Z$, there is an equivariant map $c_Z\col Z\to EG$,
unique up to equivariant homotopy.
If we set $g=c_X$ and $\psi=c_Y$
in the generalized lifting-extension problem, it can be proved
that $[X,Y]_{EG}^\emptyset$ is in a bijective correspondence with
equivariant maps $X\to Y$ up to equivariant homotopy.
This is how we obtain
Theorem~\ref{theorem:equivariant}.\footnote{Note that we cannot simply
take $B$ to be a point in the lifting-extension problem with a nontrivial $G$,
since there is no free action of $G$ on a point. Actually, $EG$
serves as an equivariant analogue of a point among free $G$-spaces.}

\item (Extending sections in a vector bundle.)
Let $G=\{e\}$, and let $\psi\col Y\to B$ be the inclusion
${BSO}(\thedim-k)\ra{BSO}(\thedim)$,
where $BSO(\thedim)$ is the classifying space of the special orthogonal group $SO(\thedim)$.
Then the commutative square in the
generalized lifting-extension
problem is essentially an oriented vector bundle of dimension $\thedim$ over $X$
together with $k$ linearly independent vector fields over $A$.
The existence of a solution is then equivalent to the existence of
linearly independent continuations of these vector fields to the whole of~$X$.
We remark that, in order to apply our theorem to this situation,
a finite simplicial model of the classifying space $BSO(\thedim)$
would have to be constructed. 
As far as we know, this has not been carried out yet.

We briefly remark that for non-oriented bundles, it is possible to pass to certain two-fold ``orientation'' coverings and reduce the problem to one for oriented bundles but with a further $\bbZ/2$-equivariance constraint.
\end{enumerate}

\subsection*{Main theorem} Now we are ready to state the main
result of this paper.

\begin{theorem} \label{thm:main_theorem}
Let $G$ be a finite group and let an instance of the generalized lifting-extension problem be input as follows: $A$, $X$, $Y$, $B$ are finite simplicial sets with free simplicial actions of $G$, $A$ is an equivariant simplicial subset of $X$, and $f$, $g$, $\psi$ are equivariant simplicial maps. Furthermore, both $B$ and $Y$ are assumed to be simply connected,
and the \emph{homotopy fibre}\footnote{The homotopy fibre of
$\psi$ is the fibre of $\psi'$, where $\psi$
is factored through $Y'$ as above. It is unique up to homotopy
equivalence, and so the connectivity is well defined.}
of $\psi\col Y\to B$ is assumed to be $d$-connected for some $d\geq 1$.

There is an algorithm that, for $\dim X\leq 2d+1$, 
 decides the existence of a solution.
Moreover, if $\dim X\leq 2d$ and a solution exists, then the set $[X,Y]_B^A$
can be equipped with the structure of an abelian group,
and the algorithm computes its isomorphism type.
\ifpoly
The running time of this algorithm is polynomial when $G$ and $\theconn$ are fixed.
\fi
\end{theorem}

As in Theorem~\ref{theorem:equivariant}, the isomorphism type means an abstract abelian group (given by generators and relations) isomorphic to $[X,Y]^A_B$.
Given an arbitrary diagonal $\ell\col X\to Y$ in the considered square, one can compute the element of this group that $\ell$ represents.

Constructing the abelian group structure on $[X,Y]_B^A$
will be one of our main objectives.
In the case of all \emph{continuous} maps $X\to Y$ up to
homotopy, with no equivariance condition imposed, as in \cite{CKMSVW11},
the abelian group structure on $[X,Y]$ is canonical.
In contrast, in the setting of the
lifting-extension problem, the structure is canonical only up to
a choice of a zero element.

This non-canonicality of zero is one of the phenomena making
the equivariant problem (and the lifting-extension problem)
substantially different from the non-equivariant case treated
in \cite{CKMSVW11}. We will have to deal with the choice of zero,
and working with ``zero sections'' in the considered fibrations.

\subsection*{Embeddability and equivariant maps}
Theorem~\ref{theorem:equivariant}
 has the following consequence
for embeddability of simplicial complexes:

\begin{theorem}\label{t:emb-metast}
\ifpoly
Let $n$ be a fixed integer.
\fi
There is an algorithm that, given a finite simplicial complex $K$ of dimension $k \leq \frac23 n-1$, decides the existence of an embedding of $K$ into $\bbR^n$%
\ifpoly
{} in polynomial time%
\fi%
.
\end{theorem}

The algorithmic problem of testing embeddability of a given $k$-dimensional
simplicial complex into $\bbR^n$, which is a natural generalization
of graph planarity,
was studied in \cite{MatousekTancerWagner:HardnessEmbeddings-2011}.
Theorem~\ref{t:emb-metast} clarifies the decidability of this problem
for $k\le \frac23 n-1$; this is
the so-called \emph{metastable range} of dimensions,
which was left open in \cite{MatousekTancerWagner:HardnessEmbeddings-2011}.
Briefly, in the metastable range, the classical theorem of Weber (see \cite{Weber:PlongementsPolyedresDomaineMetastable-1967})
asserts that embeddability is equivalent to the existence of a $\Z/2$-equivariant map $(K\times K)\smallsetminus \Delta_K \to S^{\thedim-1}$ whose domain is equivariantly homotopy equivalent to a finite simplicial complex%
\footnote{%
	The complex is (the canonical triangulation of) the union of all products $\sigma \times \tau$ of disjoint simplices $\sigma,\,\tau \in K$, $\sigma \cap \tau = \emptyset$.}
with a free simplicial action of $\Z/2$. Thus, Theorem~\ref{t:emb-metast}
follows immediately from Theorem~\ref{theorem:equivariant};
we refer to \cite{MatousekTancerWagner:HardnessEmbeddings-2011} for
details.

We also remark that the algorithm of Theorem~\ref{t:emb-metast} does not produce an actual map $(K \times K) \smallsetminus \Delta_K \to S^{\thedim-1}$ and, thus, we do not know of an effective way of producing an actual embedding (in addition, we have not analyzed Weber's proof sufficiently well to be able to tell whether it produces an embedding from an equivariant map).

\subsection*{Outline of the proof}

In the rest of this section, we sketch the main ideas and tools needed for the algorithm of Theorem~\ref{thm:main_theorem}. Even though the computation is very similar in its nature to that of \cite{CKMSVW11}, there are several new ingredients which we had to develop in order to make the computation possible. We describe these briefly after the outline of the proof.

Our first tool is a Moore--Postnikov tower $\Pnew$ for $\psi \col Y\ra B$ within the framework of (equivariant) effective algebraic topology (essentially, this means that all objects are representable in a computer); it is enough to construct the number of stages equal to the dimension of $X$. It can be shown that $[X,Y]^A_B\cong[X,\Pnew]^A_B$ for $\thedim \geq \dim X$ and so it suffices to compute inductively $[X,\Pnew]^A_B$ from $[X,\Pold]^A_B$ for $\thedim \leq \dim X$. This is the kind of problems considered in obstruction theory. Namely, there is a natural map $[X,\Pnew]^A_B\to[X,\Pold]^A_B$ and it is possible to describe all preimages of any given homotopy class $[\ell] \in [X,\Pold]^A_B$ using, in addition, an inductive computation of $[\stdsimp{1}\times X,\Pold]^{(\partial\stdsimp{1}\times X)\cup(\stdsimp{1}\times A)}_B$. In general however, $[X,\Pold]^A_B$ is infinite and it is thus impossible to compute $[X,\Pnew]^A_B$ as a union of preimages of all possible homotopy classes $[\ell]$ (on the other hand, if these sets are finite, the above description does provide an algorithm, probably not very efficient, see \cite{Brown, odd-spheres}).

For this reason, we use in the paper to a great advantage our second tool, an abelian group structure on the set $[X,\Pnew]^A_B$ of homotopy classes of diagonals, which only exists on a stable part $\thedim\leq 2\theconn$ and, of course, only if this set is non-empty. The group structure comes from an ``up to homotopy'' abelian group structure on $\Pnew$ (or, in fact, a certain pullback of $\Pnew$) which we construct algorithmically -- this is the heart of the present paper. We remark that the abelian group structure on $[X,\Pnew]^A_B$ was already observed in \cite{McClendon}; however, this paper did not deal with algorithmic aspects.

In the stable part of the Moore--Postnikov tower, the natural map $[X,\Pnew]^A_B\to[X,\Pold]^A_B$ is a group homomorphism and the above mentioned computation of preimages of a given homotopy class $[\ell]$ may be reduced to a \emph{finite} set of generators of the image; the computation is conveniently summarized in a long exact sequence \eqref{eq:les}. This finishes the rough description of our inductive computation.

\subsection*{New tools}

In the process of building the Moore--Postnikov tower, and also later, it is important to work with infinite simplicial sets, such as the Moore--Postnikov stages $P_n$, in an algorithmic way. This is handled by the so-called \emph{equivariant} effective algebraic topology and effective homological algebra. The relevant non-equivariant results are described in \cite{SergerGenova,polypost}. In many cases, only minor and/or straightforward modifications are needed. One exception is the equivariant effective homology of Moore--Postnikov stages, for which we rely on a separate paper \cite{Vokrinek}.

Compared to our previous work \cite{CKMSVW11}, the main new ingredient is the weakening of the \Hopf space structure that exists on Moore--Postnikov stages. This is needed in order to carry out the whole computation algorithmically. Accordingly, the construction of this structure is much more abstract. In \cite{CKMSVW11}, we had $B=*$ and Postnikov stages carried a unique basepoint. In the case of nontrivial $B$, the basepoints are replaced by sections and Moore--Postnikov stages may not admit a section at all -- this is related to the possibility of $[X,Y]^A_B$ being empty. It might also happen that we choose a section of $\Pold$ which does not lift to $\Pnew$. In that case, we need to change the section of $\Pold$ and compute $[X, \Pold]^A_B$ again from scratch.

\subsection*{Plan of the paper}

In the second section, we give an overview of equivariant effective homological algebra that we use in the rest of the paper.
The third section is devoted to the algorithmic construction of an equivariant Moore--Postnikov tower.
The proofs of Theorems~\ref{theorem:equivariant} and~\ref{thm:main_theorem}%
\ifpoly
, without their polynomial time claims,
\fi
are given in the following section, although proofs of its two important ingredients are postponed to Sections~$5$ and $6$.
In the fifth section, we construct a certain weakening of an (equivariant and fibrewise) \Hopf space structure on pointed stable stages of Moore--Postnikov towers.
In the sixth section, we show how this structure enables one to endow the sets of homotopy classes with addition in an algorithmic way. Finally, we derive an exact sequence relating $[X,\Pnew]^A_B$ to $[X,\Pold]^A_B$ and $[\stdsimp{1}\times X,\Pold]^{(\partial\stdsimp{1}\times X)\cup(\stdsimp{1}\times A)}_B$ and thus enabling an inductive computation.
In the seventh section, we provide proofs that we feel would not fit in the previous sections.
\ifpoly
In the last section, we prove polynomial bounds for the running time of our algorithms.
\fi

\section{Equivariant effective homological algebra} \label{sec:equi_eff_hlgy_alg}

\heading{Basic setup}
For a simplicial set, the face operators are denoted by $d_\thedimm$, and the degeneracy operators by $s_\thedimm$. The standard $\theotherdim$-simplex $\stdsimp\theotherdim$ is a simplicial set with a unique non-generate $\theotherdim$-simplex and no relations among its faces. The simplicial subset generated by the $\thedimm$-th face of $\stdsimp\theotherdim$ will be denoted by $\bdry_\thedimm\stdsimp\theotherdim$. The boundary $\partial\stdsimp\theotherdim$ is the union of all these faces and the $\thedimm$-th horn $\horn\theotherdim\thedimm$ is generated by all faces $\bdry_\thedimmm\stdsimp\theotherdim$, $\thedimmm\neq\thedimm$. Finally, we denote the vertices of $\stdsimp\theotherdim$ by $0,\ldots,\theotherdim$.

Sergeraert et al.\ (see \cite{SergerGenova})
have developed an ``effective version'' of homological algebra, in which
a central notion is an object (simplicial set or chain complex) with
effective homology. Here we will discuss analogous notions in the equivariant
setting, as well as some other extensions. For a key result,
we rely on a separate paper \cite{Vokrinek} which shows, roughly speaking,
that if the considered action is free, equivariant effective homology can be obtained from non-equivariant
one.

We begin with a description of the basic computational objects,
sometimes called \emph{locally effective} objects.
The underlying idea is that in every definition one replaces sets
by computable sets and mappings by computable mappings.
For us, a \emph{computable set} will be a set whose elements have a finite encoding
by bit strings, so that they can be represented in a computer.
On the other hand, it may happen that
no ``global'' information about the set is available; e.g.\ 
it is  algorithmically undecidable in general whether a given
computable set is nonempty.
A \emph{computable subset} of a computable set $T$ is a subset $S \subseteq T$ equipped with an algorithm that decides, for a given element of $T$, whether it belongs to $S$.
A \emph{mapping} between computable sets is \emph{computable} if there is
an algorithm computing its values.

We will need two particular cases of this principle -- simplicial sets and chain complexes.

\heading{Simplicial sets}
A \emph{locally effective simplicial set}
is a simplicial set $X$ whose simplices have a specified finite encoding and whose face and degeneracy operators are specified by algorithms. Our simplicial sets will be equipped with a simplicial \emph{action} of a finite group $G$ that is also computed by an algorithm (whose input is an element of $G$ and a simplex of $X$). We will assume that this action is \emph{free} and that a distinguished set of representatives of orbits is specified -- such $X$ will be called \emph{$G$-cellular}. In the locally effective context, we require that there is an algorithm that expresses each simplex $x\in X$ (necessarily in a unique way) as $x=ay$ where $a\in G$ and $y\in X$ is a distinguished simplex.

\begin{remark}
We will not put any further restrictions on the representation of simplicial sets in a computer -- the above algorithms will be sufficient. On the other hand, it is important that such representations exist. We will describe one possibility for finite simplicial sets and complexes.

Let $X$ be a finite simplicial set with a free action of $G$. Let us choose arbitrarily one simplex from each orbit of the non-degenerate simplices; these simplices together with all of their degeneracies are the distinguished ones. Then every simplex $x\in X$ can be represented uniquely as $x=as_Iy$, where $a\in G$, $s_I$ is an iterated degeneracy operator (i.e.\ a composition $s_{\thedimm_\theotherdim}\cdots s_{\thedimm_1}$ with $\thedimm_1<\cdots<\thedimm_\theotherdim$), and $y$ is a non-degenerate distinguished simplex. With this representation, it is possible to compute the action of $G$ and the degeneracy operators easily, while face operators are computed using the relations among the face and degeneracy operators and a table of faces of non-degenerate distinguished simplices. This table is finite and it can be provided on the input.

A special case is that of a finite simplicial complex. Here, one can prescribe a simplex (degenerate or not) uniquely by a finite sequence of its vertices.
\end{remark}

\heading{Chain complexes}
For our computations, we will work with nonnegatively graded chain complexes $C_\ast$ of abelian groups on which $G$ acts by chain maps; denoting by $\ZG$ the integral group ring of $G$, one might equivalently say that $C_*$ is a chain complex of $\ZG$-modules. We will adopt this terminology from now on. We will also assume that these chain complexes are \emph{$\ZG$-cellular}, i.e.\ equipped with a distinguished $\ZG$-basis; this means that for each $\thedim\geq 0$ there is a collection of distinguished elements of $C_\thedim$ such that the elements of the form $ay$, with $a\in G$ and $y$ distinguished, are all distinct and form a $\Z$-basis of $C_\thedim$.

In the locally effective version, we assume that the elements of the chain complex have a finite encoding, and there is an algorithm expressing arbitrary elements as (unique) $\ZG$-linear combinations of the elements of the distinguished bases.
We require that the operations of zero, addition, inverse, multiplication by elements of $\ZG$, and differentials are computable.\footnote{These requirements (with the exception of the differentials) are automatically satisfied when the elements of the chain complex are represented directly as $\ZG$-linear combinations of the distinguished bases.}

A basic example, on which these assumptions are modelled, is that of the normalized chain complex $C_*X$ of a simplicial set $X$ (the quotient of the usual chain complex by the subcomplex spanned by degenerate simplices): for each $\thedim\geq 0$, a $\Z$-basis of $C_\thedim X$ is given by the set of nondegenerate $\thedim$-dimensional simplices of $X$. If $X$ is equipped with a free simplicial action of $G$, then this induces an action of $G$ on $C_*X$ by chain maps, and a $\ZG$-basis for each $C_\thedim X$ is given by a collection of nondegenerate distinguished $\thedim$-dimensional simplices of $X$, one from each $G$-orbit.

If $X$ is locally effective as defined above, then so is $C_*X$ (for evaluating the differential, we observe that a simplex $x$
is degenerate if and only if $x=s_\thedimm d_\thedimm x$ for some $\thedimm$, and this can be checked algorithmically).

\begin{convention}\label{con:G_action_loc_eff}
We fix a finite group $G$. All simplicial sets are
locally effective, equipped with a free action of $G$ and $G$-cellular in the locally effective sense.
All chain complexes are non-negatively graded locally effective chain complexes
of free $\ZG$-modules that are moreover $\ZG$-cellular in the locally effective sense.

All simplicial maps, chain maps, chain homotopies,
etc.\ are equivariant and computable.
\end{convention}

Later, Convention~\ref{con:fibrewise} will introduce additional standing assumptions.

\begin{definition}
An \emph{effective} chain complex is a (locally effective) chain complex
equipped with an
algorithm that generates a list of elements
of the distinguished basis in any given dimension
(in particular, the distinguished bases are finite in each dimension).
\end{definition}

For example, if a simplicial set $X$ admits an algorithm generating a (finite) list of its non-degenerate
distinguished simplices in any given dimension%
\ifpoly
{} (we call it \emph{effective} in Section~\ref{sec:polynomiality})%
\fi%
, then its normalized chain complex $C_*X$ is effective.

\heading{Reductions, strong equivalences}\label{s:effReduction}

We recall that a \emph{reduction} (also called \emph{contraction} or \emph{strong deformation retraction}) $C_*\Ra C'_*$ between two chain complexes
is a triple $(\alpha,\beta,\eta)$ such that $\alpha\col C_*\ra C'_*$
and $\beta\col C'_*\ra C_*$ are equivariant chain maps such that $\alpha\beta=\id$ (i.e.\ $\beta$ is an inclusion with retraction $\alpha$) and $\eta$ is an equivariant chain homotopy on $C_*$ with $\partial\eta+\eta\partial=\id-\beta\alpha$ (i.e.\ $\eta$ is a deformation of $C_*$ onto $C'_*$); moreover, we require that $\eta\beta=0$, $\alpha\eta=0$ and $\eta\eta=0$. The following diagram illustrates this definition:

\[(\alpha,\beta,\eta)\col C_*\Ra C'_*\quad\equiv\quad\xymatrix@C=30pt{
C_* \ar@(ul,dl)[]_{\eta} \ar@/^/[r]^\alpha & C'_* \ar@/^/[l]^{\beta}
}\]

Reductions are used to solve homological problems in $C_*$ by translating them to $C_*'$ and vice versa, see \cite{SergerGenova}; a particular example is seen at the end of the proof of Lemma~\ref{l:lift_ext_one_stage}. While, for this principle to work, chain homotopy equivalences would be enough, they are not sufficient for the so-called perturbation lemmas (we will introduce them later), where the real strength of reductions lies.

For the following definition, we consider pairs $(C_*,D_*)$, where
$C_*$ is a chain complex and $D_*$ is a subcomplex of $C_*$.
Such pairs are always understood in the \emph{$\ZG$-cellular} sense; i.e.\ 
the distinguished basis of each $D_\thedim$ is a subset
of the distinguished basis of $C_\thedim$.

\begin{definition}
A \emph{reduction} $(C_*,D_*)\Ra(C'_*,D'_*)$ \emph{of ($\ZG$-cellular) pairs} is a reduction $C_*\Ra C'_*$ that restricts to a reduction $D_*\Ra D'_*$, i.e.\ such that $\alpha(D_*)\subseteq D'_*$, $\beta(D'_*)\subseteq D_*$, and $\eta(D_*)\subseteq D_*$.
\end{definition}

From this reduction, we get an induced reduction $C_*/D_*\Ra C'_*/D'_*$ of the quotients.

We will need to work with a notion more
general than reductions, namely strong equivalences.
A \emph{strong equivalence} $C_*\LRa C'_*$ is a pair of reductions
$C_*\La\widehat C_*\Ra C'_*$, where $\widehat C_*$ is some chain complex.
Similarly, a strong equivalence $(C_*,D_*)\LRa(C'_*,D'_*)$
is a pair of reductions $(C_*,D_*)\La(\widehat C_*,\widehat D_*)\Ra(C'_*,D'_*)$.
Strong equivalences can be (algorithmically) composed:
if $C_*\LRa C'_*$ and $C'_*\LRa C''_*$, then one obtains $C_*\LRa C''_*$
(see e.g.\ \cite[Lemma~2.7]{polypost}).

\begin{definition}
Let $C_*$ be a chain complex. We say that $C_*$ is equipped with \emph{effective homology} if there is specified a strong equivalence $C_*\LRa C_*^\ef$ of $C_*$ with some effective chain complex $C_*^\ef$. Effective homology for pairs $(C_*,D_*)$ of chain complexes is introduced similarly using strong equivalences of pairs. A simplicial set $X$ is equipped with \emph{effective homology} if $C_*X$ is. Finally, a pair $(X,A)$ of simplicial sets is equipped with \emph{effective homology} if $(C_*X,C_*A)$ is.
\end{definition}

\begin{remark}
In what follows, we will only assume $(X,A)$, $Y$, $B$ to be equipped with effective homology. Consequently, it can be seen that Theorems~\ref{theorem:equivariant} and~\ref{thm:main_theorem} also hold under these weaker assumptions. The dimension restriction on $X$ can be weakened to: the equivariant cohomology groups of $(X,A)$, defined in Section~\ref{sec:equiv_cohomology}, vanish above dimension $2\theconn$.

By passing to the mapping cylinder $X'=(\stdsimp1\times A)\cup X$, we may even relax the condition on the pair $(X,A)$ to each of $A$, $X$ being equipped with effective homology separately since then the pair $(X',A)$ has effective homology (this is very similar to but easier than Proposition~\ref{p:effective_homotopy_colimits}) and the resulting generalized lifting-extension problem is equivalent to the original one.
\end{remark}

The following theorem shows that, in order to equip a chain complex
with effective homology, it suffices to have it equipped
with effective homology in the non-equivariant sense.

\begin{theorem}[\cite{Vokrinek}]\label{t:vokrinek}
Let $C_*$ be a chain complex (of free $\ZG$-modules). Suppose that, as a chain complex of abelian groups, $C_*$ can be equipped with effective homology (i.e.\ in the non-equivariant sense). Then it is possible to equip $C_*$ with effective homology in the equivariant sense. This procedure is algorithmic.
\end{theorem}

The original strong equivalence $C_*\LRa C_*^\ef$ gets replaced by an equivariant one $C_*\LRa BC_*^\ef$, where $BC_*^\ef$ is a bar construction of some sort; see \cite{Vokrinek} for details.

Thus, although non-equivariant effective homology is not the same as equivariant effective homology, it is possible to construct one from the other. In this paper, effective homology will be understood in the equivariant sense, unless stated otherwise.

We recall that the \emph{Eilenberg--Zilber reduction} is a particular reduction $C_*(X\times Y)\Ra C_*X\otimes C_*Y$; see e.g.\ \cite{EML54,polypost,SergerGenova}.
It is known to be functorial (see e.g.\ \cite[Theorem~2.1a]{EML54}), and hence it is equivariant. We extend it to pairs.

\begin{proposition}[Product of pairs]\label{prop:relative_product}
If pairs $(X,A)$ and $(Y,B)$ of simplicial sets
are equipped with effective homology, then it is also possible to equip the pair
\[(X,A)\times(Y,B)\defeq\big(X\times Y,(A\times Y)\cup(X\times B)\big)\]
with effective homology.
\end{proposition}

\begin{proof}
The Eilenberg--Zilber reduction $C_*(X\times Y)\Ra C_*X\otimes C_*Y$ is functorial, which implies that it restricts to a reduction
\[C_*\big((A\times Y)\cup(X\times B)\big)\Ra(C_*A\otimes C_*Y)+(C_*X\otimes C_*B)\defeq D_*.\]
The strong equivalences $C_*X\LRa C_*^\ef X$ and $C_*Y\LRa C_*^\ef Y$ induce a strong equivalence (by \cite[Proposition~61]{SergerGenova}, whose construction is functorial, and hence applicable to the equivariant setting)
\[C_*X\otimes C_*Y\LRa C_*^\ef X\otimes C_*^\ef Y\]
that, again, restricts to a strong equivalence of the subcomplex $D_*$ above with its obvious effective version $D_*^\ef$. The composition of these two strong equivalences finally yields a strong equivalence $C_*((X,A)\times(Y,B))\LRa(C_*^\ef X\otimes C_*^\ef Y,D_*^\ef)$.
\end{proof}

Important tools, allowing us to work efficiently with reductions, are
two \emph{perturbation lemmas}.
Given a reduction $C_*\Ra C'_*$, they provide a way of obtaining
a new reduction, in which the differentials of the complexes
$C_*$, $C'_*$ are ``perturbed''.
Again, we will need versions for pairs.

\begin{definition} Let $C_*$ be a chain complex with a differential $\partial$. A collection of morphisms $\delta\col C_\thedim\to C_{\thedim-1}$ is called a \emph{perturbation} of the differential $\partial$ if the sum $\partial+\delta$ is also a differential.
\end{definition}

Since there will be many differentials around, we will emphasize them in the notation.

\begin{proposition}[Easy perturbation lemma]\label{p:epl}
Let $(\alpha,\beta,\eta)\col(C_*,D_*,\partial)\Ra(C'_*,D'_*,\partial')$ be a reduction
and let $\delta'$ be a perturbation of the differential $\partial'$ on $C'_*$ satisfying $\delta'(D_*')\subseteq D_*'$. Then $(\alpha,\beta,\eta)$ also constitutes a reduction $(C_*,D_*, \partial+\beta\delta'\alpha)\Ra(C'_*,D'_*, \partial'+\delta')$.
\end{proposition}

\begin{proposition}[Basic perturbation lemma]\label{p:bpl}
Let $(\alpha,\beta,\eta)\col(C_*,D_*,\partial)\Ra(C'_*,D'_*,\partial')$ be a reduction
and let $\delta$ be a perturbation of the differential $\partial$ on $C_*$ satisfying $\delta(D_*)\subseteq D_*$. Assume that for every $c\in C_*$ there is a $\nu\in\mathbb N$ such that $(\eta\delta)^{\nu}(c)=0$. Then it is possible to compute a perturbation $\delta'$ of the differential $\partial'$ on $C'_*$ and a reduction $(\alpha',\beta',\eta')\col(C_*,D_*,\partial+\delta)\Ra(C'_*,D'_*, \partial'+\delta')$.
\end{proposition}

The absolute versions (i.e.\ versions where all considered subcomplexes are zero) of the perturbation lemmas
are due to \cite{Shih}. There are explicit formulas provided there for $\delta'$ etc.\ (see also \cite{SergerGenova}), which show that the resulting reductions are equivariant (since all the involved maps are equivariant). Similarly, these formulas show that in the presence of subcomplexes $D_*$ and $D_*'$, these are preserved by all the maps in the new reductions (since all the involved maps preserve them).

The following proposition is used for the construction of the Moore--Postnikov tower in Section~\ref{sec:Moore_Postnikov}. Here $Z_{\thedim+1}(C_*)$ denotes
the group of all cycles in $C_{\thedim+1}$.

\begin{proposition} \label{prop:projectivity}
Let $C_*$ be an effective chain complex such that $H_\thedimm(C_*)=0$ for $\thedimm\le\thedim$. Then there is a (computable)
retraction $C_{\thedim+1}\ra Z_{\thedim+1}(C_*)$, i.e.\ 
a homomorphism that restricts to the identity on $Z_{\thedim+1}(C_*)$.
\end{proposition}

\begin{proof}
We construct a contraction%
\footnote{We recall that a contraction is a map $\sigma$ of degree $1$ satisfying $\partial\sigma+\sigma\partial=\id$.}
$\sigma$ of $C_*$ by induction on the dimension,
and use it for splitting $Z_{\thedim+1}(C_*)$ off
$C_{\thedim+1}$. It suffices to define $\sigma$ on
the distinguished bases.
Since every basis element $x\in C_0$ is a cycle,
it must be a boundary. We compute some $y\in C_1$ for which $x=\partial y$,
and we set $\sigma(x)=y$; since $G$ is finite, we may treat $\partial\col C_1\ra C_0$ as a $\bbZ$-linear map between finitely generated free $\bbZ$-modules and solve for $y$ using Smith normal form.

Now assume that $\sigma$ has been constructed up to dimension $\thedimm-1$
in such a way that $\partial\sigma+\sigma\partial=\id$, and we want to
define $\sigma(x)$ for a basis element $x\in C_\thedimm$.
Since $x-\sigma(\partial x)$ is a cycle, we can compute
some $y$ with $x-\sigma(\partial x)=\partial y$, and set
$\sigma(x)=y$.

This finishes the inductive construction of~$\sigma$.
The desired retraction $C_{\thedim+1}\ra Z_{\thedim+1}(C_*)$ is
given by $\id-\sigma\partial$.
\end{proof}

\heading{Eilenberg--MacLane spaces and fibrations}\label{sec:equiv_cohomology}

For an abelian group $\pi$, there is a  simplicial abelian group $K(\pi,\thedim+1)$, whose $\theotherdim$-simplices are the normalized ($\thedim+1$)-cocycles on $\stdsimp \theotherdim$, i.e.\ $K(\pi,\thedim+1)_\theotherdim=Z^{\thedim+1}(\stdsimp{\theotherdim},\pi)$. It is a standard model for the Eilenberg--MacLane space. We will also need a standard model for its path space, which is the simplicial abelian group $E(\pi,\thedim)_\theotherdim=C^{\thedim}(\stdsimp{\theotherdim},\pi)$ of normalized cochains. The coboundary operator $\delta\col E(\pi,\thedim)\to K(\pi,\thedim+1)$ is a fibration with fibre $K(\pi,\thedim)$.

The Eilenberg--MacLane spaces are useful for their relation to cohomology. Here we only summarize the relevant results, details may be found in \cite[Section~24]{May:SimplicialObjects-1992} or \cite[Section~3.7]{polypost} (both in the non-equivariant setup though).

When $\pi$ is a $\ZG$-module, there is an induced action of $G$ on both $K(\pi,\thedim)$ and $E(\pi,\thedim)$. We note that, in contrast to our general assumption, this action is \emph{not free} and consequently, these spaces may not possess effective homology. This will not matter since they will not enter our constructions on their own but as certain principal twisted cartesian products, see \cite{May:SimplicialObjects-1992} for the definition. Firstly, $K(\pi,\thedim)$ possesses non-equivariant effective homology by \cite[Theorem~3.16]{polypost}. The principal twisted cartesian product $P = Q \times_\tau K(\pi,\thedim)$ has a free $G$-action whenever $Q$ does and \cite[Corollary~12]{Filakovsky} constructs the non-equivariant effective homology of $P$ from that of $Q$ and $K(\pi,\thedim)$. Theorem~\ref{t:vokrinek} then provides (equivariant) effective homology for $P$.

It is easy to see that the addition in the simplicial abelian groups $K(\pi,\thedim)$, $E(\pi,\thedim)$ and the homomorphism $\delta$ between them are equivariant. Moreover, for every simplicial set $X$, there is a natural isomorphism
\[\map(X,E(\pi,\thedim))\cong C^\thedim(X;\pi)^G\]
between equivariant simplicial maps and equivariant cochains, that sends $f\col X\ra E(\pi,\thedim)$ to $f^*(\ev)$, where $\ev \in C^\thedim(E(\pi,\thedim);\pi)^G$ is the canonical cochain that assigns to each $\thedim$-simplex of $E(\pi,\thedim)_\thedim$, i.e.\ an $\thedim$-cochain on $\stdsimp\thedim$, its value on the unique non-degenerate $\thedim$-simplex of $\stdsimp{\thedim}$.

The set $\map(X,E(\pi,\thedim))$ is naturally an abelian group, with addition inhereted from that on $E(\pi,\thedim)$, and the above isomorphism is and isomorphism of groups.

When $X$ is finite, this isomorphism is computable (objects on both sides are given by a finite amount of data). When $X$ is merely locally effective, then
an algorithm that computes a simplicial map
$X\ra E(\pi,\thedim)$ can be converted into an algorithm
that evaluates the corresponding cochain in $C^\thedim(X;\pi)^G$,
and vice versa.

The above isomorphism restricts to an isomorphism
\[\map(X,K(\pi,\thedim))\cong Z^\thedim(X;\pi)^G.\]
We will denote the cohomology groups of $C^*(X;\pi)^G$ by $H^*_G(X;\pi)$.%
\footnote{Our groups $H_G^*(X;\pi)$ are the equivariant cohomology groups of $X$ with coefficients in a certain system associated with $\pi$ (see the remark in \cite[Section~I.9]{Bredon:EquivariantCohomology-1967}) or, alternatively, they are the cohomology groups of $X/G$ with local coefficients specified by $\pi$.}
We have an induced isomorphism
\[[X,K(\pi,\thedim)]\cong H^n_G(X;\pi)\]
between homotopy classes of equivariant maps and these cohomology groups. By the naturality of these isomorphisms, the maps which are zero on $A$ correspond precisely to relative cocycles and consequently
\[[(X,A),(K(\pi,\thedim),0)]\cong H^n_G(X,A;\pi).\]

\heading{Constructing diagonals for Eilenberg--MacLane fibrations}
When solving the generalized lifting-extension problem,
we will replace $\psi\col Y\to B$ by a fibration built inductively from
Eilenberg--MacLane fibrations $\delta\col E(\pi,\thedim)\to K(\pi,\thedim+1)$.
The following lemma will
serve as an inductive step in the computation of $[X,Y]^A_B$.
It also demonstrates how effective homology of pairs enters the game.

\begin{lemma} \label{l:lift_ext_one_stage}
There is an algorithm that, given a commutative square
\[\xymatrix@C=30pt{
A \ar[r]^-c \ar@{ >->}[d] & E(\pi,\thedim) \ar@{->>}[d]^\delta \\
X  \ar[r]_-z \ar@{-->}[ru] & K(\pi,\thedim+1)
}\]
where the pair $(X,A)$ is equipped with effective homology,
decides whether a diagonal exists. If it does, it computes one.

If $H^{\thedim+1}_G(X,A;\pi)=0$, then a diagonal exists for
every $c$ and $z$.
\end{lemma}

Let us remark that although our main result, Theorem~\ref{thm:main_theorem},
assumes $X$ finite, we will need to use the lemma
for infinite simplicial sets~$X$, and then the effective homology
assumption for $(X,A)$ is important.

\begin{proof}
Thinking of $c$ as a cochain in $C^\thedim(A;\pi)^G$, we extend it to a cochain on $X$ by mapping all $\thedim$-simplices not in $A$ to zero. This prescribes a map $\widetilde c \colon X\ra E(\pi,\thedim)$ that is a solution of the lifting-extension problem from the statement for $z$ replaced by $\delta\widetilde c$. Since the lifting-extension problems and their solutions are additive, one may subtract this solution from the previous problem
and obtain an equivalent lifting-extension problem
\[\xymatrix@C=30pt{
A \ar[r]^-{0} \ar@{ >->}[d] & E(\pi,\thedim) \ar@{->>}[d]^\delta \\
X  \ar[r]_-{z-\delta\widetilde c} \ar@{-->}[ru]^-{c_0} & K(\pi,\thedim+1)
}\]
A solution of this problem is an (equivariant) relative cochain $c_0$ whose coboundary is $z_0=z-\delta\widetilde c$ (this $c_0$ yields a solution $\widetilde c+c_0$ of the original problem). If $C_*(X,A)$ is effective, then such a $c_0$ is computable whenever it exists (and it always exists in the case $H^{\thedim+1}_G(X,A;\pi)=0$).

However, $C_*(X,A)$ itself is not effective in general,
it is only strongly equivalent to an effective complex.
Thus, we need to check that the computability of a preimage under $\delta$
is preserved under reductions in both directions.
Let $(\alpha,\beta,\eta)\col C_*\Ra C'_*$ be a reduction. First,
let us suppose that $z_0'\col C'_*\ra\pi$ is a cocycle
with $z_0'\alpha=\delta c_0$. Then
\[z_0'=z_0'\alpha\beta=(\delta c_0)\beta=\delta(c_0\beta),\]
and we may set $c_0'=c_0\beta$. Next, suppose that $z_0\col C_*\ra\pi$ is
a cocycle with $z_0\beta=\delta c_0'$. Then
\[z_0=z_0(\partial\eta+\eta\partial+\beta\alpha)=z_0\eta\partial+\delta c_0'\alpha=\delta(z_0\eta+c_0'\alpha),\]
and we may set $c_0=z_0\eta+c_0'\alpha$.
\end{proof}

\section{Moore--Postnikov tower}\label{sec:Moore_Postnikov}

We recall that we defined $Y'$ by factoring $\psi$ as a composition $Y\cof[\sim] Y'\fib[\psi']B$ of a weak homotopy equivalence followed by a fibration; such a factorization exists by Lemma~\ref{l:fibrant_replacement}. Using this approximation, $[X,Y]^A_B$ was defined as the set of homotopy classes $[X,Y']^A_B$. In order to compute this set, we approximate $Y'$ by the Moore--Postnikov tower of $Y$ over $B$. Then the computation will proceed by induction over the stages of this tower, as will be explained in Section~\ref{s:main_proofs}. For now, we give a definition of an equivariant Moore--Postnikov tower of a simplicial map $\psi\col Y\to B$ and review some of the statements of the last section in the context of this tower. The actual construction of the tower, when both simplicial sets $Y$ and $B$ are equipped with effective homology, will be carried out later in Section~\ref{sec:MP_tower_proof}.

\begin{definition}
Let $\psi\col Y\to B$ be a map. A (simplicial)
\emph{extended Moore--Postnikov tower} for $\psi$ is a commutative diagram
\[\xymatrix{
& & {} \ar@{.}[d] \\
& & \Pnew \ar[d]^-{\pn} \ar@/^2.5pc/[ddd]^-{\psin} \\
& & \Pold \ar@{.}[d] \\
Y \ar[uurr]^-{\varphin} \ar[urr]_-{\varphi_{\thedim-1}} \ar[rr]_-{\varphi_1} \ar[drr]_-{\rightbox{\scriptstyle\psi={}}{\scriptstyle\varphi_0}}
& & P_{1} \ar[d]^-{p_1} \\
& & \leftbox{P_{0}}{{}=B}
}\]
satisfying the following conditions:
\begin{enumerate}[labelindent=.5em,leftmargin=*,label=\arabic*.]
\renewcommand{\theenumi}{\arabic{enumi}}
\item\label{MP1}
The induced map $\varphinst\col \pi_\thedimm(Y)\to \pi_{\thedimm}(\Pnew)$ is an isomorphism for $\thedimm\leq\thedim$ and an epimorphism for $\thedimm=\thedim+1$.

\item\label{MP2}
The induced map $\psinst \colon \pi_\thedimm(\Pnew)\to\pi_\thedimm(B)$ is an isomorphism for $\thedimm\ge\thedim+2$ and a monomorphism for $\thedimm=\thedim+1$.

\item\label{MP3}
The map $\pn\col\Pnew\to\Pold$ is a Kan fibration induced by a map
\[\knp\col\Pold\to K(\pin,\thedim+1)\]
for some $\ZG$-module $\pin$, i.e.\ there exists a pullback square
\[\xymatrix@C=30pt{
\Pnew \pb \ar[r]^-\qnp \ar[d]_-\pn & E(\pin,\thedim) \ar[d]^{\delta} \\
\Pold \ar[r]_-\knp & K(\pin,\thedim+1)
}\]
identifying $\Pnew$ with the pullback $\Pold\times_{K(\pin,\thedim+1)}E(\pin,\thedim)$. Alternatively, one may identify $\Pnew$ as the principal twisted cartesian product $\Pnew\times_\tau K(\pin,\thedim)$ -- this will be used to equip $\Pnew$ with effective homology.
\end{enumerate}

A \emph{Moore--Postnikov tower} for $\psi$ is then obtained from the extended Moore--Postnikov tower by removing the space $Y$ and the maps $\varphin$.

Both variants admit \emph{$\theDim$-truncated} versions comprised only of stages $\Pnew$ with $\thedim \leq \theDim$.
\end{definition}

We remark that the axioms imply $\pin\cong\pin F$, where $F$ is the homotopy fibre of $Y\ra B$, i.e.\ the fibre of $Y'\to B$.

\begin{definition}
We say that an extended Moore--Postnikov tower has \emph{effective homology} if $Y$ and all the stages $\Pnew$ have effective homology and all the maps $\varphin$, $\pn$, $\qnp$, $\knp$ are computable. There are similar notions for a Moore--Postnikov tower and for $\theDim$-truncated versions of both variants.
\end{definition}

We remark that it is also possible to compute the homotopy groups $\pin$ from the effective homology of a Moore--Postnikov tower as homology groups $H_{n+1}(\cone \pnst)$ of the mapping cone of $\pnst \colon C_*(\Pnew) \to C_*(\Pold)$, see the proof of Theorem~\ref{t:MP_tower}.

The reason to have various versions of Moore--Postnikov towers is to specify the objects that we construct, equip with effective homology etc. Concretely, a Moore--Postnikov tower for $\psi \colon Y \to B$ is also a Moore--Postnikov tower for the replacement $\psi' \colon Y' \to B$. They are different as extended Moore--Postnikov towers and, in fact, we will be able to equip the former with effective homology, while we do not know of a way of doing the same for the latter (because of the space $Y'$). Another example is Addendum~\ref{a:MP_tower_pullback}.

\newcommand{\MPt}
{There is an algorithm that, given a map $\psi\col Y\to B$ between simply connected simplicial sets with effective homology and an integer $\theDim$, constructs an $\theDim$-truncated extended Moore--Postnikov tower for $\psi$ and equips it with effective homology.}
\begin{theorem}\label{t:MP_tower}
\MPt
\end{theorem}

The proof of the theorem, as well as its addendum below, is postponed to Section~\ref{sec:MP_tower_proof}.

\begin{addendum}\label{a:MP_tower_pullback}
There is an algorithm that, given the data of the theorem and a computable map $\beta \colon \tB \to B$ whose domain $\tB$ has effective homology, constructs an $\theDim$-truncated Moore--Postnikov tower with stages $\tPnew = \tB \times_B \Pnew$ and equips it with effective homology.
\end{addendum}

We remark that the $\tPnew$ form a Moore--Postnikov tower for the natural map $\widetilde Y = \widetilde B \times_B Y' \to \widetilde B$ from the homotopy pullback $\widetilde Y$ of $Y$ along $\beta$, but we do not know of a way of dealing effectively with $\widetilde Y$. This is the reason why we are \emph{not} able to equip the \emph{extended} Moore--Postnikov tower for $\widetilde Y \to \tB$ with effective homology.

We obtain a new lifting-extension problem from the Moore--Postnikov tower for $\psi$
\[\xymatrix@C=40pt{
A \ar[r]^-{\fn} \ar@{ >->}[d] & \Pnew \ar@{->>}[d]^-{\psin} \\
X  \ar[r]_-g \ar@{-->}[ru] & B
}\]
where $\fn=\varphin f$. The following theorem explains the role of the Moore--Postnikov tower in our algorithm.

\newcommand{\nequiv}
{There exists a map $\varphinp\col Y'\to\Pnew$ inducing a bijection $\varphinpst\col[X,Y']^A_B\to [X,\Pnew]^A_B$ for every $\thedim$-dimensional simplicial set $X$ with a free action of $G$.}
\begin{theorem}\label{t:n_equivalence}
\nequiv
\end{theorem}

The theorem should be known but we could not find an equivariant fibrewise version anywhere. For this reason, we include a proof in Section~\ref{sec:n_equivalence_proof}.

From the point of view of Theorem~\ref{thm:main_theorem}, we have reduced the computation of $[X,Y]^A_B=[X,Y']^A_B$ to that of $[X,\Pnew]^A_B$, where $\thedim=\dim X$. Before going into details of this computation, we present a couple of results that are directly related to the Moore--Postnikov tower. They will be essential tools in the proof of Theorem~\ref{thm:main_theorem}.

\heading{Inductive construction of diagonals}

We slightly reformulate Lemma~\ref{l:lift_ext_one_stage} in terms of the Moore--Postnikov tower in the following proposition, which works for stages of a Moore--Postnikov tower.

\begin{proposition} \label{prop:lift_ext_one_stage}
There is an algorithm that, given a diagram
\[\xymatrix@C=40pt{
A \ar[r]^-f \ar@{ >->}[d] & \Pnew \ar@{->>}[d]^-{\pn} \\
X  \ar[r]_-g \ar@{-->}[ru] & \Pold
}\]
where the pair $(X,A)$ is equipped with
effective homology, decides whether a diagonal exists.
If it does, it computes one.

When $H^{\thedim+1}_G(X,A;\pin)=0$, a diagonal exists for every
 $f$  and~$g$.
\end{proposition}

\begin{proof}
We will use property (3) of Moore--Postnikov towers, which expresses $\pn$ as a pullback:
\[\xymatrix@C=30pt{
A \ar[r]^-f \ar@{ >->}[d] & \Pnew \pb \ar[r] \ar[d]_-{\pn} & E(\pin,\thedim) \ar@{->>}[d]^\delta \\
X  \ar[r]_-g \ar@{-->}[ru]^{\ell} & \Pold \ar[r]_-{\knp} & K(\pin,\thedim+1)
}\]
Thus, diagonals $\ell$ are exactly of the form
$(g,c)\col X\ra \Pold\times_{K(\pin,\thedim+1)}E(\pin,\thedim)$,
where $c\col X\to E(\pin,\thedim)$
is an arbitrary diagonal in the composite square
 and thus computable by Lemma~\ref{l:lift_ext_one_stage}.
\end{proof}

We obtain two important consequences as special cases. The first one is an algorithmic version of lifting homotopies across $\Pnew\fib\Pm$.

\begin{proposition}[homotopy lifting/extension] \label{prop:homotopy_lifting}
Given a diagram
\[\xymatrix{
(\vertex\thedimm\times X)\cup(\stdsimp{1}\times A) \ar[r] \ar@{ >->}[d]_-\sim & \Pnew \ar@{->>}[d] \\
\stdsimp{1}\times X  \ar[r] \ar@{-->}[ru] & P_\theotherdim
}\]
where $\thedimm\in\{0,1\}$ and $(X,A)$ is equipped with
effective homology, it is possible to compute a diagonal.
In other words, one may lift and extend homotopies in Moore--Postnikov towers
algorithmically.
\end{proposition}

\begin{proof}
It is possible to equip $(\stdsimp{1}\times X,
(\vertex\thedimm\times X)\cup(\stdsimp{1}\times A))$
with effective homology by Proposition~\ref{prop:relative_product}.
Moreover, this pair has zero cohomology since there exists
a (continuous) equivariant deformation of $\stdsimp{1}\times X$
onto the considered subspace. Thus a diagonal can be
constructed by a successive use of Proposition~\ref{prop:lift_ext_one_stage}.
\end{proof}

The second result concerns algorithmic concatenation of homotopies.
Let $\horn21$ denote the first horn in the standard $2$-simplex $\stdsimp2$, i.e.\ the simplicial subset of the standard simplex $\stdsimp2$
spanned by the faces $\bdry_2\stdsimp2$ and $\bdry_0\stdsimp2$.
Given two homotopies $h_2,h_0\col\stdsimp1\times X\to Y$
that are compatible, in the sense that $h_2$ is a homotopy from $\ell_0$ to $\ell_1$ and $h_0$ is a homotopy from $\ell_1$ to $\ell_2$,
one may prescribe a map $\horn21\times X\to Y$ as $h_2$ on
$\bdry_2\stdsimp2\times X$ and as $h_0$ on $\bdry_0\stdsimp2\times X$.
This map has an extension
$H\col \stdsimp2\times X\to Y$ and the restriction of $H$ to
 $\bdry_1\stdsimp2\times X$ gives a homotopy from $\ell_0$ to $\ell_2$,
which can be thought of as a \emph{concatenation} of $h_2$ and $h_0$.
We will need the following effective,
relative and fibrewise version; the proof is entirely analogous
to that of the previous proposition and we omit it.

\begin{proposition}[homotopy concatenation] \label{prop:homotopy_concatenation}
Given a diagram
\[\xymatrix{
(\horn{2}{1}\times X)\cup(\stdsimp{2}\times A) \ar[r] \ar@{ >->}[d]_-\sim & \Pnew \ar@{->>}[d] \\
\stdsimp{2}\times X  \ar[r] \ar@{-->}[ru] & P_\theotherdim
}\]
where $(X,A)$ is equuipped with effective homology, it is possible to
compute a diagonal. In other words, one may concatenate homotopies
in Moore--Postnikov towers algorithmically.
\end{proposition}

\section{Computing homotopy classes of maps}\label{s:main_proofs}

In this section, we prove Theorems~\ref{theorem:equivariant} and~\ref{thm:main_theorem}. First, we explain our computational model for abelian groups, since these are one of our main computational objects and also form the output of our algorithms.

There are two levels of these computational models: 
\emph{semi-effective} and \emph{fully effective} abelian groups.
They are roughly analogous  to locally effective chain complexes and effective ones. There is, however, one significant difference: while an element
of a chain complex is assumed to have a unique computer representation,
a single element of a semi-effective abelian group may have many
different representatives. We can perform the group operations
in terms of the representatives but, in general, we cannot decide whether
two representatives represent the same group element.
This setting is natural
when working with elements of $[X,\Pnew]^A_B$, i.e.\ homotopy classes of diagonals. The representatives are simplicial maps
$X\to\Pnew$, and at first, we will not be able to decide whether
two given such maps are homotopic. 

Given a semi-effective abelian group, it is not possible to compute its isomorphism type (even when it is finitely generated); for this we need additional information, summarized in the notion of a fully effective abelian group. 
A semi-effective abelian group can be made fully effective provided
that it is a part of a suitable exact sequence, additionally provided
with set-theoretic sections; this is described in Lemma~\ref{l:ses}.

This suggests a computation of $[X,\Pnew]^A_B$ in two steps. First, in Theorem~\ref{t:semi_eff}, we endow it with a structure of a semi-effective abelian group (whose addition comes from the weak \Hopf space structure on $\Pnew$ constructed later in Section~\ref{sec:H_space_constr}). Next, we promote it to a fully effective abelian group by relating it to $[X,\Pold]^A_B$ and $[\stdsimp 1 \times X,\Pold]^{(\partial\stdsimp 1 \times X) \cup (\stdsimp 1 \times A)}_B$ through a long exact sequence of Theorem~\ref{thm:exact_sequence_long} and using induction.

We note that the long proofs of Theorems~\ref{t:semi_eff}~and~\ref{thm:exact_sequence_long} are postponed to later sections. This enables us to complete the proof of the main Theorem~\ref{thm:main_theorem} in the present section.

\heading{Operations with abelian groups}\label{s:abelops}

This subsection is a short summary of a detailed discussion found in \cite{CKMSVW11}; results not included there are proved.

In our setting, an abelian group $A$ is represented by a set $\mcA$, whose elements are called \emph{representatives}; we also assume that the representatives have a finite encoding by bit strings. For $\alpha\in\mcA$, let $[\alpha]$ denote the element of $A$ represented by $\alpha$. The representation is generally non-unique; we may have $[\alpha]=[\beta]$ for $\alpha\ne\beta$.

We call $A$ represented in this way \emph{semi-effective}, if algorithms for the following three tasks are available: provide an element $o\in\mcA$ with $[o]=0$ (the neutral element); given $\alpha\comma\beta\in\mcA$, compute $\gamma\in\mcA$ with $[\gamma]=[\alpha]+[\beta]$; given $\alpha\in\mcA$, compute $\beta \in\mcA$ with $[\beta]=-[\alpha]$.

For semi-effective abelian groups $A$, $B$, with sets $\mcA$, $\mcB$ of representatives, respectively, we call a mapping $f\col A\to B$ \emph{computable} if there is a computable mapping $\varphi\col\mcA\to\mcB$ such that $f([\alpha])=[\varphi(\alpha)]$ for all $\alpha\in \mcA$.

We call a semi-effective abelian group $A$ \emph{fully effective} if there is given an isomorphism $A\cong\bbZ/q_1\oplus\cdots\oplus\bbZ/q_r$, computable together with its inverse. In detail, this consists of \begin{itemize}
\item
a finite list of generators $a_1,\ldots,a_r$ of $A$ (given by representatives) and their orders $q_1,\ldots,q_r\in\{2,3,\ldots\}\cup\{0\}$ (where $q_\thedimm=0$ gives $\bbZ/q_\thedimm=\bbZ$),
\item
an algorithm that, given $\alpha\in\mcA$, computes integers $z_1,\ldots,z_r$ so that $[\alpha]=\sum_{i=1}^rz_i a_i$; each coefficient $z_i$ is unique within $\bbZ/q_i$.
\end{itemize}

The proofs of the following lemmas are not difficult. The first is \cite[Lemma~3.2 and~3.3]{CKMSVW11}.

\begin{lemma}[kernel and cokernel]\label{l:ker_coker}
Let $f\col A\to B$ be a computable homomorphism
of fully effective abelian groups.
Then both $\ker(f)$ and $\coker(f)$ can be represented as fully effective abelian groups.
\end{lemma}

This implies formally that the same holds for $\im(f)$, since it equals the kernel of the projection $B \to \coker(f)$.

\begin{nexample}\label{e:homology}
Clearly, every chain group $C_n$ in an effective chain complex $C_*$ is fully effective. Thus, so are the subgroups of cocyles $Z_n(C_*)$ and boundaries $B_n(C_*)$ and, consequently, also the homology groups $H_n(C_*) = Z_n(C_*) / B_n(C_*)$. The same applies to cohomology groups of effective cochain complexes.
\end{nexample}

\begin{definition}
A \emph{semi-effective exact sequence} (of abelian groups) is an exact sequence
\[\cdots \lra A_{n+1} \xlra{d_{n+1}} A_n \xlra{d_n} A_{n-1} \xlra{d_{n-1}} A_{n-2} \lra \cdots\]
of semi-effective abelian groups and computable homomorphisms such that the induced maps
\[d_n \colon \coker d_{n+1} \to \ker d_{n-1}\]
have computable inverses, called \emph{sections}. If the sequence is bounded from either side, we require sections only for inner differentials.

Since $A_n / \ker d_n$ is represented by $\mcA_n$ and $\im d_n$ by a subset of $\mcA_{n-1}$, this amounts to computable partial mappings $\rho_{n-1} \colon \mcA_{n-1} \pto \mcA_n$, defined on representatives of $\im d_n$, such that $d_n[\rho_{n-1}(\gamma)] = [\gamma]$. In general, it may happen that $[\gamma] = [\gamma']$, while $[\rho_{n-1}(\gamma)] \neq [\rho_{n-1}(\gamma')]$.
\end{definition}

\begin{lemma}[5-lemma]\label{l:ses}
There is an algorithm that, given a semi-effective exact sequence
\[A_{2} \xlra{d_2} A_{1} \xlra{d_1} A_0 \xlra{d_0} A_{-1} \xlra{d_{-1}} A_{-2},\]
with all $A_{-2}$, $A_{-1}$, $A_1$ and $A_2$ fully effective, makes also $A_0$ fully effective.
\end{lemma}

\begin{proof}
Consider the induced short exact sequence
\[0 \lra \coker d_2 \xlra{d_1} A_0 \xlra{d_0} \ker d_{-1} \lra 0.\]
Viewing sections as maps from the kernel to the cokernel, it is still a semi-effective exact sequence. Now apply \cite[Lemma~3.5]{CKMSVW11}.
\end{proof}

\begin{definition}
We say that a mapping $f \colon A \to B$ between groups is an \emph{affine homomorphism} if its translate $f^0 \colon A \to B$, given by $f^0(a) = f(a) - f(0)$, is a group homomorphism. This is equivalent to
\begin{equation}\label{eq:affine_homomorphism}
f(a + b) = f(a) + f(b) - f(0)
\end{equation}
\end{definition}

Clearly, for semi-effective $A$ and $B$, an affine homomorphism $f$ is computable iff $f^0$ and the constant $f(0)$ are computable. We will also need the following simple lemma.

\begin{lemma}[preimage]\label{l:preimage}
Let $f \colon A \to B$ be a computable affine homomorphism of fully effective abelian groups. Then there is an algorithm that, given $b \in B$, decides whether it lies in $\im f$. If it does, it computes a preimage $a \in f^{-1}(b)$.
\end{lemma}

\begin{proof}
Equivalently, we ask for $f^0(a) = b - f(0)$. Compute the images $f^0(a_1), \ldots, f^0(a_r)$ of the generators of $A$. Next, decide if the equation
\[x_1 f^0(a_1) + \cdots + x_r f^0(a_r) = b - f(0)\]
has a solution (this is done by translating to the direct sum of cyclic groups and solving there using standard methods). If a solution exists, output $a = x_1 a_1 + \cdots + x_r a_r$.
\end{proof}

\heading{Making Eilenberg--MacLane spaces fibrewise}
The description of $\Pnew$ in the definition of a Moore--Postnikov tower as a pullback is both classical and useful for the actual construction of the tower. For the upcoming computations, it has a major disadvantage though -- the spaces appearing in the pullback square are not spaces over $B$. This is easily corrected by replacing the Eilenberg--MacLane space by the product $\Kn=B\times K(\pin,\thedim+1)$ and the ``path space'' by $\En=B\times E(\pin,\thedim)$. Denoting by $\kn$ the fibrewise Postnikov invariant, i.e.\ the map whose first component is the projection $\psiold\col\Pold\to B$ and the second component is the original (non-fibrewise) Postnikov invariant $\knp$, we obtain another pullback square
\[\xymatrix@C=40pt{
\Pnew \pb \ar[r]^-{\qn} \ar[d]_-{\pn} & E_\thedim \ar[d]^-\delta \\
\Pold \ar[r]_-\kn & K_{\thedim+1}
}\]

We will need that $\Ln=B\times K(\pin,\thedim)$ is a fibrewise abelian group: for two elements $z=(b,z')$ and $w=(b,w')$ of $\Ln$ lying over the same $b\in B$, we define $z+w\defeq(b,z'+w')$. The same applies to $\En$ and $\Kn$.

Since we know that homotopy classes of maps into Eilenberg--MacLane spaces correspond to cohomology groups and these are easy to compute, the following result should not be surprising; in its statement, the fixed map $A \to \Ln$ is the only fibrewise map (over $B$) with values on the zero section, i.e.\ $(g \iota, 0) \colon A \to B \times K(\pin, \thedim)$; we call it the \emph{zero map} and denote it $0$.

\begin{lemma}\label{l:fully_eff_cohlgy}
Let $(X,A)$ be equipped with effective homology. Then it is possible to equip $[X,\Ln]^A_B$ with a structure of a fully effective abelian group; the elements are represented by algorithms that compute (equivariant) fibrewise simplicial maps $X\to\Ln$ that take $A$ to the zero section.
\end{lemma}

\begin{proof}
We start with isomorphisms
\[[X,\Ln]^A_B\cong[(X,A),(K(\pi,\thedim),0)]\cong H^\thedim_G(X,A;\pi)\cong H^\thedim_G(X,A;\pi)^\ef\]
where the group on the right is the cohomology group of the ``effective'' cochain complex $C^*_\ef(X,A;\pi)^G=\operatorname{Hom}_\ZG(C_*^\ef(X,A),\pi)$ of equivariant cochains on the effective chain complex of $(X,A)$; the last isomorphism comes from effective homology of $(X,A)$.

Elements of these groups are represented by algorithms that compute the respective (equivariant) simplicial maps or equivariant cocycles and it is possible to transform one such representing algorithm into another, so that the isomorphisms are computable in both directions. The last group is fully effective by Example~\ref{e:homology}.
\end{proof}

It will also be useful to generalize the above lemma to the case of maps whose restriction to $A$ is fixed to a non-zero map. For practical reasons, we will formulate this for $[X, \Kn]^A_B$ and will assume that the fixed restriction is of the form $\delta c$ for some fibrewise map $c \colon A \to \En$.

\begin{lemma}\label{l:fully_eff_cohlgy_snd}
Let $(X,A)$ be equipped with effective homology. Then it is possible to equip $[X,\Kn]^A_B$ with a structure of a fully effective abelian group; the elements are represented by algorithms that compute (equivariant) fibrewise simplicial maps $X\to\Kn$ whose restriction to $A$ equals $\delta c$.
\end{lemma}

\begin{proof}
We denote the group from the statement $[X, \Kn]^{A,c}_B$ and start with the computation of its zero. Namely, it is possible to compute an extension $\widetilde c \colon X \to \En$ as in the proof of Lemma~\ref{l:lift_ext_one_stage}. The zero is then represented by $\delta \widetilde c$. There is an isomorphism
\[[X, \Kn]^{A,0}_B \xlra\cong [X, \Kn]^{A,c}_B, \quad [\ell] \mapsto [\ell + \delta\widetilde c],\]
computable in both directions. The group on the left has been endowed with a fully effective abelian group structure in Lemma~\ref{l:fully_eff_cohlgy}.
\end{proof}

We remark that the homotopy class of the zero is independent of the choice of $\widetilde c$: it is the only homotopy class in the image of $\delta_* \colon [X, \En]^A_B \to [X, \Kn]^A_B$ -- the domain has a single element since $\En$ is (fibrewise) contractible. We denote this homotopy class $0 = [\delta \widetilde c]$.

\subsection*{Semi-effectiveness of $\boldsymbol{[X,\Pnew]^A_B}$ for stable stages $\boldsymbol{\Pnew}$} \label{s:semi_eff_str}

\begin{definition}
We call a Moore--Postnikov stage $\Pnew$ \emph{stable} if $\thedim\leq 2\theconn$, where $\theconn$ is the connectivity of the homotopy fibre of $\psi\col Y\to B$ (as in the introduction).
\end{definition}

We remark that $d$ is also the connectivity of the homotopy fibre of $\psin \colon \Pnew \to B$ and, thus, stability may be defined without any reference to $Y$.

The significance of the stability condition lies in the existence of an abelian group structure on $[X,\Pnew]^A_B$. The construction of this structure is (together with the construction of the Moore--Postnikov tower) technically the most demanding part of the paper and we postpone it to later sections. For its existence, we will have to assume that $[X,\Pnew]^A_B$ is non-empty; in fact, the structure depends on the choice of a zero of this group, i.e.\ an element $[\onew] \in [X,\Pnew]^A_B$.

\newcommand{\se}
{Suppose that $\Pnew$ is a stable stage of a Moore--Postnikov tower with effective homology and that $(X,A)$ is equipped with effective homology. Then, for any given solution $\onew \colon X \to \Pnew$, the set $[X,\Pnew]^A_B$ admits a structure of a semi-effective abelian group with zero $[\onew]$, whose elements are represented by algorithms that compute diagonals $X\to \Pnew$.}
\begin{theorem}\label{t:semi_eff}
\se
\end{theorem}

The proof of the theorem occupies a significant part of the paper. First, we construct a ``weak \Hopf space structure'' on $\Pnew$ (or, in fact, a pullback of it) in Section~\ref{s:weak_H_spaces} and then show how this structure gives rise to addition on the homotopy classes of diagonals in Section~\ref{sec:weak_H_space}.

\heading{Exact sequence relating consecutive stable stages} \label{sec:exact_sequence}

To promote the semi-effective group structure on $[X,\Pnew]^A_B$ to a fully effective one, we will apply Lemma~\ref{l:ses} to a certain exact sequence relating two consecutive stable stages of the Moore--Postnikov tower. The sequence involves the groups $[X,\Ln]^A_B$ and $[X,\Kn]^A_B$, where the fixed restrictions are the zero map $A \to \Ln$ and the composite $\delta \qn \fn \colon A \to \Kn$.

\newcommand{\exseqshort}{
Suppose that $\thedim\leq 2\theconn$ and that $(X,A)$ is equipped with effective homology. For any given zero $[\oold] \in [X, \Pold]^A_B$, the computable map $\knst$ in
\[[X, \Pnew]^A_B \xlra{\pnst} [X, \Pold]^A_B \xlra{\knst} [X, \Kn]^A_B\]
is an affine homomorphism and $\im \pnst = \knstinv(0)$.}
\begin{theorem} \label{thm:exact_sequence_short}
\exseqshort
\end{theorem}

In the next theorem, a given zero $[\onew] \in [X, \Pnew]^A_B$ induces naturally, for all $i \leq n$, zeros $[o_i] \in [X, P_i]^A_B$ and Theorem~\ref{t:semi_eff} then provides $[X, P_i]^A_B$ with a group structure. Further, the group $[\stdsimp 1 \times X,P_i]^{(\partial\stdsimp 1 \times X) \cup (\stdsimp 1 \times A)}_B$ consists of homotopy classes of homotopies $o_i \sim o_i$ relative to $A$ (this prescribes the fixed restriction to the subspace $(\partial\stdsimp 1 \times X) \cup (\stdsimp 1 \times A)$), whose zero is the homotopy class of the \emph{constant} homotopy at $o_i$.

\newcounter{les}
\newcommand{\exseqseq}{
[\stdsimp 1 \times X,\Pold]^{(\partial\stdsimp 1 \times X) \cup (\stdsimp 1 \times A)}_B & \xlra{\connn} [X,\Ln]^A_B\xlra{\jnst} \notag \\
\xlra{\jnst}[X,\Pnew]^A_B & \xlra{\pnst} [X,\Pold]^A_B \xlra{\knst} [X,\Kn]^A_B}
\newcommand{\exseqlong}{
Suppose that $\thedim\leq 2\theconn$, that $(X,A)$ is equipped with effective homology and that a zero $[\onew] \in [X, \Pnew]^A_B$ is given in such a way that $[\stdsimp 1 \times X, \Polder]^{(\partial\stdsimp 1 \times X) \cup (\stdsimp 1 \times A)}_B$ is fully effective for all $\thedimm < \thedim -1$. Then there is a semi-effective exact sequence
\ifthenelse{\equal{\theles}{0}}{
\begin{align}\label{eq:les}
\exseqseq
\end{align}\setcounter{les}{1}}{
\begin{align*}
\exseqseq
\end{align*}}%
of abelian groups.}
\begin{theorem} \label{thm:exact_sequence_long}
\exseqlong
\end{theorem}

The exactness itself ought to be well known and is nearly \cite[Proposition~II.2.7]{Crabb_James}.
The proofs are postponed to Section~\ref{sec:proof_of_exact_sequence}.

\heading{Proof of Theorem~\ref{thm:main_theorem}} \label{sec:description}

Let us review the reductions made so far. By Theorem~\ref{t:n_equivalence}, it is enough to compute $[X,\Pnew]^A_B$ for $\thedim=\dim X\leq 2\theconn$. The rest of the proof does not depend on the dimension of $X$. Concretely, we prove the following two claims for all pairs $(X, A)$ with effective homology by induction with respect to $\thedim \leq 2 \theconn$:
\begin{enumerate}[labelindent=.5em,leftmargin=*,itemsep=0pt,parsep=0pt,topsep=5pt]
\item 
	given a zero $[\onew] \in [X,\Pnew]^A_B$, make $[X, \Pnew]^A_B$ into a fully effective abelian group;
\item
	decide if $[X, \Pnew]^A_B$ is non-empty and, if this is the case, compute an element $[\onew]$.
\end{enumerate}
Since $P_0 = B$, we have $[X,P_0]^A_B = *$ and both claims are trivial in this case.

By Theorem~\ref{t:semi_eff}, $[X,\Pnew]^A_B$ is a semi-effective abelian group. According to Theorem~\ref{thm:exact_sequence_long}, this group fits into an exact sequence with all remaining terms fully effective either by Lemma~\ref{l:fully_eff_cohlgy}, Lemma~\ref{l:fully_eff_cohlgy_snd} or by induction, since they concern diagonals into $\Pold$ (the domain $\stdsimp 1 \times X$ of the leftmost term admits effective homology by Proposition~\ref{prop:relative_product}). Lemma~\ref{l:ses} makes $[X, \Pnew]^A_B$ fully effective.

If $[X, \Pold]^A_B$ is empty, so is $[X, \Pnew]^A_B$. Otherwise, compute a zero of $[X, \Pold]^A_B$ and make it into a fully effective abelian group structure. Next, use Lemma~\ref{l:preimage} to decide if $0$ lies in the image of the affine homomorphism $\knst$ and, if this is the case, compute a preimage $[\oold]$ (generally different from the chosen zero of $[X, \Pold]^A_B$). Finally, lift $\oold$ to $\onew \colon X \to \Pnew$ using Proposition~\ref{prop:lift_ext_one_stage} -- a lift exists by Theorem~\ref{thm:exact_sequence_short}.

\subsection*{Deciding existence for $\thedim = \dim X = 2\theconn+1$}
Since Lemma~\ref{l:lift_ext_one_stage} guarantees the existence of a diagonal $X\to\Pnew$ as a lift of \emph{any} partial diagonal $X\to\Pold$, it is enough to decide whether the stable $[X,\Pold]^A_B$ is non-empty.\qed

\heading{Proof of Theorem~\ref{theorem:equivariant}}

We describe how the set of equivariant homotopy classes of maps $[X,Y]$ between two $G$-simplicial sets can be computed as a particular stable instance of the lifting-extension problem, namely $[X,Y]^\emptyset_{EG}$, so that Theorem~\ref{thm:main_theorem} applies.

This instance is obtained by setting $B=EG$, where $EG$ (known as the Rips complex) is a non-commutative version of $E(\pi,0)$. It has as $\thedim$-simplices sequences $(a_0,\ldots,a_\thedim)$ of elements $a_\thedimm\in G$, and its face and degeneracy operators are the maps
\begin{align*}
d_\thedimm(a_0,\ldots,a_\thedim) & =(a_0,\ldots,a_{\thedimm-1},a_{\thedimm+1},\ldots,a_\thedim) \\
s_\thedimm(a_0,\ldots,a_\thedim) & =(a_0,\ldots,a_{\thedimm-1},a_\thedimm,a_\thedimm,a_{\thedimm+1},\ldots,a_\thedim).
\end{align*}
There is an obvious diagonal action of $G$ which is clearly free.

As every $k$-simplex of $EG$ is uniquely determined by its (ordered) collection of vertices, it is clear that a simplicial map $g\col X\ra EG$ is uniquely determined by the mapping $g_0\col X_0\ra G$ of vertices and $g$ is equivariant if and only if $g_0$ is. A particular choice of a map $X\to EG$ is thus uniquely specified by sending the distinguished vertices of $X$ to $(e)$; it is clearly computable. Moreover, any two equivariant maps $X\to EG$ are (uniquely) equivariantly homotopic (vertices of $\stdsimp1\times X$ are those of $\vertex0\times X$ and $\vertex1\times X$).

Factoring $Y\to EG$ as $Y\cof[\sim]Y'\fib EG$ using Lemma~\ref{l:fibrant_replacement}, the geometric realization of $Y'$ equivariantly deforms onto that of $Y$. This shows that the first map in
\[[X,Y]\xra\cong[X,Y']\la[X,Y']^\emptyset_{EG}\]
is a bijection and it remains to study the second map. As observed above, for every simplicial map $X\to Y'$, the lower triangle in
\[\xymatrix@C=40pt{
\emptyset \ar[r] \ar[d] & Y' \ar@{->>}[d] \\
X \ar[r] \ar@{-->}[ur]^-\ell & EG
}\]
commutes up to homotopy. Since $Y'\fib EG$ is a fibration, one may replace $\ell$ by a homotopic map for which it commutes strictly, showing surjectivity of $[X,Y']^\emptyset_{EG}\to[X,Y']$. The injectivity is implied by uniqueness of homotopies -- every homotopy of maps $X\to Y'$ that are diagonals is automatically vertical.

It remains to show how to identify a given equivariant map $\ell\col X\to Y$ as an element of the computed group $[X,\Pnew]^\emptyset_{EG}$. By its fully effective abelian group structure, it is enough to find the corresponding diagonal $X\to\Pnew$. As above, compute a homotopy $h$ from $\psi\ell$ to $g\col X\to EG$; then, using Proposition~\ref{prop:homotopy_lifting}, compute a lift of $h$ that fits into
\[\xymatrix{
\vertex0\times X \ar[r]^-{\ell} \ar@{ >->}[d] & Y \ar[r]^-{\varphin} & \Pnew \ar@{->>}[d] \\
\stdsimp1\times X \ar[rr]_-h \ar@{-->}[rru]^{\widetilde h} & & EG
}\]
The restriction of $\widetilde h$ to $\vertex1\times X$ is the required diagonal $X\to\Pnew$.
\qed
\vskip\topsep

We remark that it is also possible to compute $[X,Y]$ as $[X, X \times Y]^\emptyset_X$.

\section{Weak \Hopf spaces}\label{s:weak_H_spaces}

Our goal for the following two sections is to equip $[X, \Pnew]^A_B$ with a semi-effective abelian group structure. We will do this indirectly -- we replace $\Pnew$, a space over $B$, by a certain pullback $\tPnew$, a space over $\tB$. Proposition~\ref{prop:homotopy_classes_pullback} will then give an isomorphism $[X, \Pnew]^A_B \cong [X, \tPnew]^A_\tB$, computable in both directions, and will thus reduce our task to a similar one for $\tPnew$. The main advantage of $\tPnew$ over $\Pnew$ is that the projection $\tpsin \colon \tPnew \to \tB$ admits a section $\tonew \colon \tB \to \tPnew$ that we may think of as a choice of a point in each fibre of $\tpsin$ (made in a ``continuous'' way) -- we say that $\tPnew$ is pointed.

This is the first step to introducing a fibrewise \Hopf space structure on $\tPnew$; again, one could think of this structure as a choice of an \Hopf space structure on each fibre that is made in a ``continuous'' way. The fibrewise \Hopf space structure on $\tPnew$ induces an abelian group structure on the set of fibrewise homotopy classes of maps to $\tPnew$ as usual; this is described in Section~\ref{sec:weak_H_space}.

To simplify the notation, i.e.\ in order to deal with $\Pnew$ rather than $\tPnew$, we will assume in this section that $\Pnew$ itself is pointed (and stable) and equip it with a fibrewise \Hopf space structure and treat the general case only in the next section.

First, we explain a simple approach to constructing
a \emph{strict} fibrewise \Hopf space structure, which we were not able to make
algorithmic, but which introduces ideas employed in the actual
proof of~Theorem~\ref{t:semi_eff}, and it also
shows why a weakening of the \Hopf space structure is needed.

We start with additional running assumptions.

\begin{convention}\label{con:fibrewise}
In addition to Convention~\ref{con:G_action_loc_eff}, all simplicial sets are equipped with a map to $B$ and all maps, homotopies, etc.\ are fibrewise, i.e.\ they commute with the specified maps to $B$. In the case of homotopies, this means that they remain in one fibre the whole time or, in other words, that they are vertical.
\end{convention}

\begin{definition}
We say that a space $P$ over $B$, with projection $\psi \colon P \fib B$, is \emph{pointed} if there is provided a section $o \colon B \to P$, i.e.\ a map such that $\psi o = \id$. We will call this distinguished section $o$ the \emph{zero section}.
\end{definition}

\heading{Fibrewise \Hopf spaces}

Let $P$ be a pointed space over $B$ with projection $\psi\col P\fib B$ and zero section $o\col B\ra P$. We recall that the pullback $P\times_BP$ consists of pairs $(x,y)$ with $\psi(x)=\psi(y)$. Associating to $(x,y)$ this common value makes $P\times_BP$ into a space over $B$. We recall that a (fibrewise) \emph{\Hopf space structure} on $P$ is a (fibrewise) map
\[\add\col P\times_BP\ra P,\]
where we write $\add(x,y)=x+y,$ that satisfies a single condition -- the zero section $o$ should act as a zero for this addition, i.e.\ for $x\in P$ lying over $b=\psi(x)$ we have $o(b)+x=x=x+o(b)$. In the proceeding, we will abuse the notation slightly and write $o$ for any value of $o$, so that we rewrite the zero axiom as $o+x=x=x+o$. After all, there is a single value of $o$ for which this makes sense. It will be convenient to organize this structure into a commutative diagram
\[\xymatrix@C=40pt{
P\vee_BP \ar[rd]^-\nabla \ar[d]_\vartheta \\
P\times_BP \ar[r]_-\add & P
}\]
with $P\vee_BP$ the fibrewise wedge sum, $P\vee_BP=(B\times_BP)\cup(P\times_BB)$ (where $B\subseteq P$ is the image of the zero section $o$), and with $\nabla$ denoting the fold map given by $(o,x)\mapsto x$ and $(x,o)\mapsto x$. As explained, all maps are fibrewise over $B$. Under this agreement, the above diagram is a \emph{definition} of a (fibrewise) \Hopf space structure.

We say that the \Hopf space structure is \emph{homotopy associative} if there exists a homotopy $(x+y)+z\sim x+(y+z)$ (i.e.\ formally a homotopy of maps $P\times_BP\times_BP\to P$) that is constant when restricted to $x=y=z=o$. \emph{Homotopy commutativity} is defined similarly. Finally, it has a \emph{right homotopy inverse} if there exists a map $\inv\col P\to P$, denoted $x\mapsto -x$, such that $-o=o$ and such that there exists a homotopy $x+(-x)\sim o$, constant when restricted to $x=o$.

We have already met an example of an \Hopf space, namely $\Ln=B\times K(\pin,\thedim)$. We recall that $\Pnew$ is a stable stage if $\thedim \leq 2 \theconn$, where $\theconn$ is the connectivity of the homotopy fibre of $\psi$. In general, we have the following theorem, whose proof can be found in Section~\ref{sec:existence_of_H_space_structures_proof}.

\newcommand{\eoHss}
{Every pointed stable Moore--Postnikov stage $\Pnew$ admits a fibrewise \Hopf space structure. Any such structure is homotopy associative, homotopy commutative and has a right homotopy inverse. It is unique up to homotopy relative to $\Pnew \vee_B \Pnew$.}
\begin{theorem}\label{t:existence_of_H_space_structures}
\eoHss
\end{theorem}

The importance of this result does not lie in the existence of an \Hopf space structure itself but in its uniqueness and its properties. After all, we will need to construct this structure and, in this respect, the above existential result is not sufficient.

\heading{\Hopf space structures on pullbacks}\label{sec:non_constr_H_space}

We describe a general method for introducing \Hopf space structures on pullbacks since $\Pnew$ is defined in this way. Let us start with a general description of our situation. We are given a pullback square
\[\xymatrix@C=40pt{
P \ar[r] \ar[d] \pb & R \ar@{->>}[d]^-\psi \\
Q \ar[r]_-\chi & S
}\]
with $\psi$ a fibration. We assume that all of $Q$, $R$ and $S$ are \Hopf spaces over $B$, and that $R$ and $S$ are strictly associative, commutative and with a strict inverse. If both $\psi$ and $\chi$ preserved the addition strictly we could define addition on $P\subseteq Q\times R$ componentwise. In our situation,
though, $\chi$ preserves the addition only up to homotopy and, accordingly, the addition on $P$ will have to be perturbed to
\begin{equation}\label{eq:addition_on_pullback}
(x,y)+(x',y')=(x+x',y+y'+M(x,x')).
\end{equation}
There are two conditions that need to be satisfied in order for this formula to be correct: $\psi M(x,x')=\chi(x+x')-(\chi(x)+\chi(x'))$ (so that the right-hand side of \eqref{eq:addition_on_pullback} lies in the pullback) and $M(x,o)=o=M(o,x)$ (to get an \Hopf space). Both are summed up in the following lifting-extension problem
\[\xymatrix@C=30pt{
P\vee_BP \ar[r]^-o \ar[d]_-\vartheta & R \ar@{->>}[d]^-\psi \\
P\times_BP \ar[r]_-m \ar@{-->}[ru]^-M & S
}\]
with $m(x,x')=\chi(x+x')-(\chi(x)+\chi(x'))$. In our situation, $\psi$ is $\delta\col\En\to\Kn$. Thus, Lemma~\ref{l:lift_ext_one_stage} would give us a solution if the pair $(P\times_BP,P\vee_BP)$ had effective homology. 

However, we have not been able to prove this and, consequently, we cannot construct the addition on the pullback. In the computational world, we are thus forced to replace this pair by a certain homotopy version $(P\htimes_BP,P\hvee_BP)$ of it that admits effective homology. This transition corresponds, as will be explained later, to a passage from \Hopf spaces to a weakened notion, where the zero section serves as a zero for the addition only up to homotopy.

After this rather lengthy introduction, the plan for the rest of the section is to introduce weak \Hopf spaces and then to describe an inductive construction of weak \Hopf space structure on pointed stable stages of Moore--Postnikov towers. We believe that to understand the weak version, it helps significantly to keep in mind the above formula for addition on~$P$. For the same reason, we give a formula for a right inverse in $P$, assuming that it exists in $Q$ (and in $R$ and $S$, as required earlier):
\begin{equation}\label{eq:inverse_on_pullback}
-(x,y)=(-x,-y-M(x,-x)).
\end{equation}

\heading{Weak \Hopf spaces} \label{s:hocolim}

We will need a weak version of an \Hopf space. Roughly speaking this is defined to be a fibrewise addition $x+y$ together with left zero and right zero homotopies $\lambda\col y\sim o+y$ and $\rho\col x\sim x+o$ that become homotopic as homotopies $o\sim o+o$. In simplicial sets, a homotopy between homotopies can be defined in various ways. Here we will interpret it as a map $\eta\col\stdsimp{2}\times B\ra P$ that is a constant homotopy on $\bdry_2\stdsimp{2}\times B$ and restricts to the two unit homotopies on $\bdry_1\stdsimp{2}\times B$ and $\bdry_0\stdsimp{2}\times B$, respectively:
\[\xymatrix{
& o+o \POS[];[d]**{}?(.6)*{//{\scriptstyle\eta}//} \\
o \ar[rr]_{s_0o} \ar[ru]^{\lambda(o)} & & o \ar[lu]_{\rho(o)}
}\]

We will organize this data into a map $\add\col P\htimes_BP\to P$ with similar properties to the strict \Hopf space structure. The space $P\htimes_BP$ will be a special case of the following construction which works for any commutative square (of spaces over $B$)
\[\mcS={}\quad\xymatrixc{
Z \ar[r]^-{u_0} \ar[d]_-{u_1} & Z_0 \ar[d]^-{v_0} \\
Z_1 \ar[r]_-{v_1} & Z_2
}\]
that we denote for simplicity by $\mcS$. We define $|\mcS|$ and its subspace $\bdry_2|\mcS|$ as particular small models of the homotopy colimit of the square $\mcS$ and of the homotopy pushout of $Z_0$ and $Z_1$ along $Z$; namely,
\begin{align*}
|\mcS| & =\big(\stdsimp{2}\times Z\big)\cup\big(\bdry_1\stdsimp{2}\times Z_0\big)\cup\big(\bdry_0\stdsimp{2}\times Z_1\big)\cup\big(2\times Z_2\big), \\
\bdry_2|\mcS| & =\big(\bdry_2\stdsimp{2}\times Z\big)\cup\big(0\times Z_0\big)\cup\big(1\times Z_1\big),
\end{align*}
where we assume for simplicity that all maps in $\mcS$ are inclusions; otherwise, the union has to be replaced by a certain (obvious) colimit. In the case of inclusions, $|\mcS|$ is naturally a subspace of $\stdsimp2\times Z_2$ and as such admits an obvious map to $\stdsimp2\times B$ whose fibres are equal to those of $Z$, $Z_0$, $Z_1$ or $Z_2$ (over $B$), depending on the point of $\stdsimp2$. In the picture below, $B=\{*\}$ and $|\mcS|$ is thus depicted as a space over $\stdsimp2$; here, $Z_2$ is a $3$-simplex, $Z_0$ and $Z_1$ its edges and $Z$ their common vertex.
\begin{figure}[H]
\centering
\includegraphics{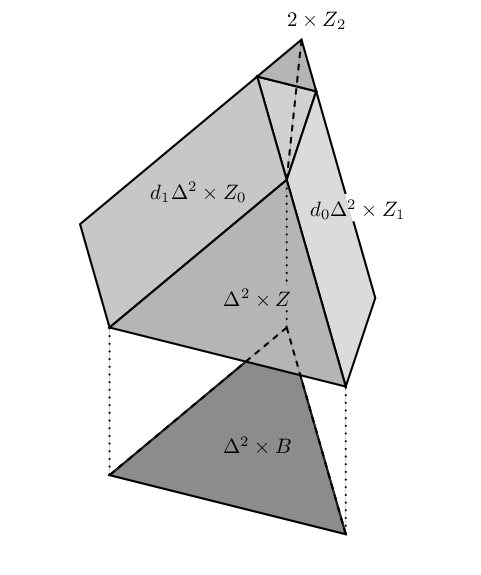}
\end{figure}\noindent

The construction $|\mcS|$ possesses the following universal property: to give a map $f\col |\mcS|\ra Y$ is the same as to give maps $f_\thedimm\col Z_\thedimm\ra Y$ (for $\thedimm=0\comma 1\comma 2$), homotopies $h_\thedimm\col f_\thedimm\sim f_2v_\thedimm$ (for $\thedimm=0\comma 1$) and a ``second order homotopy'' $H\col \stdsimp{2}\times Z\ra Y$ whose restriction to $\bdry_\thedimm\stdsimp{2}\times Z$ equals $h_\thedimm u_\thedimm$ (for $\thedimm=0\comma 1$). Similarly, a map $\bdry_2|\mcS|\ra Y$ is specified by $f_0\col Z_0\to Y$ and $f_1\col Z_1\to Y$ as above and a homotopy $f_0u_0 \sim f_1u_1$.

In order to apply this definition to weak \Hopf spaces, we consider the square
\[\mcS_P={}\qquad\xymatrixc{
B\times_BB \ar@{ >->}[r] & \leftbox{B\times_BP}{{}\defeq P_\mathrm{right}} \ar@{ >->}[d] \\
\rightbox{P_\mathrm{left}\defeq{}}{P\times_BB} \ar@{ >->}[r] \POS[]*+!!<0pt,\the\fontdimen22\textfont2>{\phantom{P\times_BB}}="a";[u]\ar@{ >->}"a" & P\times_BP
}\]
(the subspaces consist of pairs where one of the two components, or both, lie on the zero section). We will denote $B\times_BB$ for simplicity by $B$, to which it is canonically isomorphic.

\begin{definition}
Let $P\to B$ be a Kan fibration. We define simplicial sets
\[P\hvee_BP\stackrel{\mathrm{def}}{=}\bdry_2|\mcS_P|,\qquad P\htimes_BP\stackrel{\mathrm{def}}{=}|\mcS_P|.\]
We denote the inclusion by $\vartheta\col P\hvee_BP\to P\htimes_BP$.

Furthermore, we define a ``fold map'' $\hnabla\col P\hvee_BP\ra P$, prescribed as the identity map on $0\times P_\mathrm{right}$ and $1\times P_\mathrm{left}$ and as the constant homotopy at $o$ on $\bdry_2\stdsimp{2}\times B$.
\end{definition}

We remark that $P\hvee_BP$ and $P\htimes_BP$ are weakly homotopy equivalent to $P\vee_BP$ and $P\times_BP$, respectively; this is proved in Lemma~\ref{l:whe_weak_strict}. Now, we are ready to define weak \Hopf spaces.

\begin{definition}
A \emph{weak \Hopf space structure} on $P$ is a (fibrewise) map $\add\col P\htimes_BP\to P$ that fits into a commutative diagram
\[\xymatrix@C=40pt{
P\hvee_BP \ar[rd]^-\hnabla \ar[d]_\vartheta \\
P\htimes_BP \ar[r]_-\add & P
}\]
We denote the part of $\add$ corresponding to $2\times(P\times_BP)$ by $x+y=\add(2,x,y)$, the part corresponding to $\bdry_1\stdsimp{2}\times P_\mathrm{right}$, i.e.\ the left zero homotopy, by $\lambda$, and the part corresponding to $\bdry_0\stdsimp{2}\times P_\mathrm{left}$, i.e.\ the right zero homotopy, by $\rho$.

Finally, we define a ``diagonal'' $\hDelta\col P\ra P\htimes_BP$ by $x\mapsto(2,x,x)$.
\end{definition}

All these associations are natural, making $P\hvee_BP$, $P\htimes_BP$ into functors and $\hnabla$, $\vartheta$, $\hDelta$ into natural transformations.

\newcommand{\ehc}
{Assume that all the spaces in the square $\mcS$ have effective homology. Then so does the pair $(|\mcS|,\bdry_2|\mcS|)$.}
\begin{proposition}\label{p:effective_homotopy_colimits}
\ehc
\end{proposition}

The proof is given in Section~\ref{sec:effective_homotopy_colimits}. The following special case will be crucial in constructing a weak \Hopf space structure on pointed stable stages of Moore--Postnikov towers.

\newcommand{\wHseh}
{Let $\Pnew$ be a pointed stage of a Moore--Postnikov tower with effective homology. Then it is possible to equip the pair $(\Pnew\htimes_B\Pnew,\Pnew\hvee_B\Pnew)$ with effective homology.}
\begin{corollary} \label{c:weak_H_space_eff_hlgy}
\wHseh
\end{corollary}

\begin{proof}
According to Addendum~\ref{a:MP_tower_pullback}, it is possible to equip $\Pnew\times_B\Pnew$ with effective homology. Thus, the result follows from the previous proposition.
\end{proof}

\begin{remark}
Alternatively, we may construct effective homology of $\Pnew\times_B\Pnew$ at the same time as we build the tower for $Y\ra B$ but, compared to $\Pnew$, with all Eilenberg--MacLane spaces and all Postnikov classes ``squared''.
\end{remark}

The following proposition will be used in Section~\ref{sec:H_space_constr} as a certificate for the existence of a weak \Hopf space structure on $\Pnew$; namely, it will guarantee that all relevant obstructions vanish.

\newcommand{\wHsc}{
For any Moore--Postnikov stage $\Pnew$, the pair $(\Pnew\htimes_B\Pnew,\Pnew\hvee_B\Pnew)$ is $(2\theconn+1)$-connected, where $d$ is the connectivity of the homotopy fibre of $\psi \col Y \ra B$ (or equivalently of $\psin \colon \Pnew \to B$).

In particular, the cohomology groups $H^*_G(\Pnew\htimes_B\Pnew,\Pnew\hvee_B\Pnew;\pi)$ of this pair with arbitrary coefficients $\pi$ vanish up to dimension $2\theconn+1$.}
\begin{proposition} \label{prop:weak_H_space_connect}
\wHsc
\end{proposition}

The proof can be found in Section~\ref{sec:leftover_proofs}.

\heading{Constructing weak \Hopf spaces}\label{sec:H_space_constr}

Prime examples of weak \Hopf spaces are the strict ones and, in particular, every fibrewise simplicial group is a weak \Hopf space. In the proceeding, we will make use of the trivial bundles $\Kn=B\times K(\pin,\thedim+1)$ and $\En=B\times E(\pin,\thedim)$. Since $\Kn$ is a fibrewise simplicial group, we have a whole family of weak \Hopf space structures on $\Kn$, one for each choice of a zero section $o\col B\ra\Kn$; namely, we define addition $z+_{o}w=z+w-o$ (the inverse then becomes $-_oz=-z+2o$). A similar formula defines an \Hopf space structure on $\En$ for every choice of its zero section. We denote the usual zero section by $0$.

We are now ready to prove the following crucial proposition.

\begin{proposition}\label{prop:weak_H_space_structure}
If $\Pnew$ is a pointed stable stage of a Moore--Postnikov tower with effective homology, with a zero section $\onew$, it is possible to construct a structure of a weak \Hopf space on $\Pnew$ with a strict right inverse.
\end{proposition}

\begin{proof}
The proof is by induction and the base case is trivial since $P_0 = B$. Let $\Pold$ be a Moore--Postnikov stage and $\kn\col\Pold\ra\Kn$ the respective (fibrewise) Postnikov invariant. There is a pullback square
\begin{equation} \label{eq:consecutive_stages}
\xy *!C\xybox{\xymatrix@C=40pt{
\Pnew \ar[r]^{\qn} \ar[d]_{\pn} \pb & \En \ar[d]^\delta \\
\Pold \ar[r]^\kn & \Kn
}}\endxy
\end{equation}
of spaces over $B$. We denote the images of the zero section $\onew\col B\to\Pnew$ by $\oold=\pn\onew$ in $\Pold$, by $\qn\onew$ in $\En$ and by $\kno$ in $\Kn$.
In this way $\Kn$ is equipped with two sections, the zero section $0$ and
the composition $\kno$. We will see that the fact that these do not
coincide in general causes some technical problems.

Assume inductively that there is given a structure of a weak \Hopf space on $\Pold$.
\[\xymatrix@C=40pt{
\Pold\hvee_B\Pold \ar[rd]^-\hnabla \ar[d]_\vartheta \\
\Pold\htimes_B\Pold \ar[r]_-\add & \Pold
}\]
In analogy with Section~\ref{sec:non_constr_H_space}, we form the ``non-additivity'' map $m\col \Pold\htimes_B \Pold\ra \Kn$ as the difference of the following two compositions
\[\xymatrix@=15pt{
& \Pold\htimes_B\Pold \ar[ld]_{\add} \ar[rd]^{\kn\htimes \kn} & \\
\Pold \ar[rd]_\kn & - & \leftbox{\Kn}{{}\htimes_B\Kn} \ar[ld]^{\add_{\kno}} \\
& \Kn
}\]
where $\add_{\kno}$ is the \Hopf space structure on $\Kn$ whose zero section is $\kno$. We recall that it is given by $z+_{\kno}w=z+w-\kno$.

We now construct a weak \Hopf space structure on $\Pnew=\Pold\times_\Kn\En$ under our stability assumption $\thedim\leq 2\theconn$. The zero of this structure will be $\onew$. We compute a diagonal in
\begin{equation}\label{eq:lift_homotopy_additivity}
\xy *!C\xybox{\xymatrix@C=30pt{
\Pold\hvee_B\Pold \ar[r]^-0 \ar[d]_-\vartheta & \En \ar[d]^-\delta \\
\Pold\htimes_B\Pold \ar[r]_-m \ar@{-->}[ru]^-M & \Kn
}}\endxy
\end{equation}
by Lemma~\ref{l:lift_ext_one_stage}, whose hypotheses are satisfied according to Corollary~\ref{c:weak_H_space_eff_hlgy} and Proposition~\ref{prop:weak_H_space_connect}. The existence of $M$ says roughly that $\kn$ is additive up to homotopy. We define
\[\mathop{\add}\col \Pnew\htimes_B\Pnew\lra \Pnew=\Pold\times_\Kn\En\]
by its two components $\pn\add$ and $\qn\add$. The first component $\pn\mathop{\add}$ is uniquely specified by the requirement that $\pn\col \Pnew\ra \Pold$ is a homomorphism, i.e.\ by the commutativity of the square
\[\xymatrix@C=30pt{
\Pnew\htimes_B\Pnew \ar[r]^-\add \ar[d]_-{\pn\htimes\pn} & \Pnew \ar[d]^-{\pn} \\
\Pold\htimes_B\Pold \ar[r]_-\add & \Pold
}\]
The second component $\qn\mathop{\add}$ is given as a sum
\[\xymatrix@=15pt{
& \Pnew\htimes_B\Pnew \ar[ld]_{\qn\htimes\qn} \ar[rd]^{\pn\htimes\pn} & \\
\rightbox{\En\htimes_B{}}{\En} \ar[rd]_{\add_{\qn\onew}} & + & \leftbox{\Pold}{{}\htimes_B\Pold} \ar[ld]^M \\
& \En \\
}\]
The last two diagrams are a ``weak'' version of the formula \eqref{eq:addition_on_pullback}. A simple diagram chase shows that the two components are compatible and satisfy the condition of a weak \Hopf space; details can be found in Lemma~\ref{lem:calculations}.

Assuming that $\inv$ is constructed on $\Pold$ in such a way that $x+(-x)=\oold$, we define a right inverse on $\Pnew$ by the formula
\[-(x,c)=(-x,-c+2\qn\onew-M(2,x,-x)).\]
Again, $\inv$ is well defined and is a right inverse for $\add$; details can be found in Lemma~\ref{lem:calculations}.
\end{proof}

\section{Structures induced by weak \Hopf spaces}\label{sec:weak_H_space}

In this section, we prove Theorems~\ref{t:semi_eff}, \ref{thm:exact_sequence_short} and~\ref{thm:exact_sequence_long}. We start by general considerations.

\begin{definition}
We say that a lifting-extension problem
\[\xymatrix{
A \ar[r]^-f \ar@{ >->}[d]_-\iota & P \ar@{->>}[d]^-\psi \\
X \ar[r]_-g \ar@{-->}[ru] & B \ar@/_1pc/@{-->}[u]_-o
}\]
is \emph{pointed} if $P$ is pointed in such a way that $f = og\iota$.
\end{definition}

This condition is equivalent to $og$ being a solution; thus, $[X, P]^A_B$ in naturally pointed by the homotopy class $[og]$. Until further notice, we consider a pointed lefting-extension problem.

In the case of a strict \Hopf space $P$ over $B$, it is easy to define addition on $[X,P]^A_B$: simply put $[\ell_0]+[\ell_1]=[\ell_0+\ell_1]$. In particular, this defines addition on $[X,\Ln]^A_B$ which, under the identification of $[X,\Ln]^A_B$ with $H^\thedim_G(X,A;\pin)$, corresponds to the addition in the cohomology group.

It is technically much harder to equip $[X,P]^A_B$ with addition when the \Hopf space structure on $P$ is \emph{weak}. In this case, the restriction of $\ell_0+\ell_1$ to $A$ equals $f+f \neq f$ (note that the values of $f$ lie on the zero section and $o + o \neq o$) and thus does not represent an element of $[X,P]^A_B$. This problem is solved in Section~\ref{sec:strictness} using a strictification of weak \Hopf space structures, which serves as a compact definition of addition in $[X,P]^A_B$ and is also a useful tool in proofs that deal with the addition in $[X,P]^A_B$ on a global level, e.g.\ in deriving the exact sequence of Theorem~\ref{thm:exact_sequence_long}.

\heading{Strictification and addition of homotopy classes} \label{sec:strictness}

The point of this subsection is to describe a perturbation of a weak \Hopf space structure to one for which the zero is strict. We will then apply this to the construction of addition on $[X,\Pnew]^A_B$. Assume thus that we have a weak \Hopf space structure
\[\xymatrix@C=40pt{
P\hvee_BP \ar[d]_-\vartheta \ar[rd]^-{\hnabla} \\
P\htimes_BP \ar[r]_-\add & P
}\]
Form the following lifting-extension problem where the top map is $\add$ on $P\htimes_BP$ and $\nabla\pr_2$ on $\bdry_2\stdsimp{2}\times(P\vee_BP)$. Lemma~\ref{l:whe_weak_strict} shows that the map on the left is a weak homotopy equivalence and thus a diagonal exists (but in general not as a computable map).
\[\xymatrix@C=50pt{
(P\htimes_BP)\cup\big(\bdry_2\stdsimp{2}\times(P\vee_BP)\big) \ar@{ >->}[d]_-\sim \ar[r]^-{[\add,\nabla\pr_2]} & P \ar@{->>}[d]^-\psi \\
\stdsimp{2}\times(P\times_BP) \ar[r]  \ar@{-->}[ur] & B
}\]
The restriction of the diagonal to $0\times(P\times_BP)$ is then a (strict!) \Hopf space structure which we denote $\hadd$ with the corresponding addition $\hplus$. The restriction to $\bdry_1\stdsimp2\times(P\times_BP)$ is a homotopy $\hplus \sim +$.

\begin{definition}
Let $P$ be a weak \Hopf space with addition $\add$. Let $\hadd$ be its perturbation to a strict \Hopf space structure as above. We define the addition in $[X,P]^A_B$ by $[\ell_0]+[\ell_1]=[\ell_0\hplus\ell_1]$. Below, we prove that it is independent of the choice of a perturbation.
\end{definition}

Composing the above homotopy $\hplus \sim +$ with a pair of solutions $(\ell_0, \ell_1)$, we obtain $\ell_0\hplus\ell_1\sim\ell_0+\ell_1$ whose restriction to $A$ is the left zero homotopy $\lambda f \col f=f\hplus f\sim f+f$. We will use this observation as a basis for the computation of the homotopy class of $\ell_0\hplus\ell_1$, since we do not see a way of computing $\hadd$ directly -- it seems to require certain pairs to have effective homology and we think that this might not be the case in general.

Restricting to the case $P=\Pnew$ of Moore--Postnikov stages, the addition in $[X,\Pnew]^A_B$ is computed in the following algorithmic way. Let $\ell_0,\ell_1\col X\to\Pnew$ be two solutions and consider $\ell_0+\ell_1$ whose restriction to $A$ equals $f+f$. Extend the left zero homotopy $\lambda f \col f\sim f+f$ on $A$ to a homotopy $\sigma\col\ell\sim\ell_0+\ell_1$ on $X$. It is quite easy to see that the resulting map $\ell$ is unique up to homotopy relative to $A$.\footnote{%
	Given two such homotopies, one may form out of them a map $(\horn22\times X)\cup(\stdsimp2\times A)\to\Pnew$, whose extension to $\stdsimp2\times X$, fibrewise over $B$, gives on $\bdry_2\stdsimp2\times X$ the required homotopy.
} Since $\ell_0\hplus\ell_1$ is also obtained in this way, this procedure gives correctly $[\ell]=[\ell_0]+[\ell_1]\in[X,\Pnew]^A_B$. From the algorithmic point of view, this is well behaved -- if $(X,A)$ is equipped with effective homology, we may extend homotopies by Proposition~\ref{prop:homotopy_lifting}. This proves the first half of the following proposition.

\begin{proposition}\label{prop:addition_on_homotopy_classes}
If $(X,A)$ is equipped with effective homology and $\Pnew$ is given a weak \Hopf space structure, then there exists an algorithm that computes, for any two solutions $\ell_0\comma\ell_1\col X\ra\Pnew$ of a pointed lifting-extension problem, a representative of $[\ell_0]+[\ell_1]$. If the weak \Hopf space structure has a strict right inverse, then the computable $\onew g+(-\ell)$ is a representative of $-[\ell]$.
\end{proposition}

\begin{proof}
The formula $\ell\mapsto\onew g+(-\ell)$ prescribes a mapping $[X,\Pnew]^A_B\ra[X,\Pnew]^A_B$ since its restriction to $A$ equals $f+(-f)=f$. It is slightly more complicated to show that it is an inverse for our perturbed version of the addition. To this end, we have to exhibit a homotopy
\[\onew g\sim\ell+(\onew g+(-\ell))\]
that agrees on $A$ with the left zero homotopy $\lambda f$. We start with the left zero homotopy $\lambda(-\ell)\col-\ell\sim\onew g+(-\ell)$ and add $\ell$ to it on the left to obtain $\ell+\lambda(-\ell)\col\onew g\sim\ell+(\onew g+(-\ell))$. By Lemma~\ref{lem:commuting_left_zero_homotopies}, its restriction to $A$, i.e.\ $f+\lambda(-f)$, is homotopic to the left zero homotopy $\lambda(f+(-f))=\lambda f$. By extending this second order homotopy from $A$ to $X$, we obtain a new homotopy $\onew g\sim \ell+(\onew g+(-\ell))$ that agrees with the left zero homotopy on $A$, as desired.
\end{proof}

To make the statement of the following lemma understandable, we use $\onew$ to denote the appropriate value of $\onew$, i.e.\ they are abbreviations for $\onew \psin(x)$. Applying to $x = -f$ as in the previous proof, this equals $\onew \psin (-f) = f$.

\begin{lemma}\label{lem:commuting_left_zero_homotopies}
The homotopies $\lambda(\onew+x),\onew+\lambda(x)\col\onew+x\sim \onew+(\onew+x)$ are homotopic relative to $\partial\stdsimp{1}\times\Pnew$.
\end{lemma}

\begin{proof}
We concatenate the two homotopies from the statement with the left zero homotopy $\lambda(x)\col x\sim \onew+x$ and it is then enough to show that the two concatenations are homotopic. The homotopy between them is $\stdsimp{1}\times\stdsimp{1}\times\Pnew\ra\Pnew$, $(s,t,x)\mapsto\lambda(s,\lambda(t,x))$.
\end{proof}

\heading{Solution of pullback problems}

So far, we have discussed only pointed Moore--Postnikov stages and pointed lifting-extension problems. We will now describe, in a general stable situation of Theorem~\ref{t:semi_eff}, a way of passing from a solution $\onew \colon X \to \Pnew$ to a pointed Postnikov stage and a pointed lifting-extension problem.

First we describe a general procedure for replacing, via pullbacks, lifting-extension problems by equivalent ones. Suppose that we have a diagram
\[\xymatrix{
A \ar[r] \ar@{ >->}[d] & \tPnew \ar[r] \ar@{->>}[d] \pb & \Pnew \ar@{->>}[d] \\
X \ar[r] & \tB \ar[r] & B
}\]
in which the right square is a pullback square. Then diagonals in the left square are in bijection with diagonals in the composite square and the same applies to homotopies. Thus,
\[[X, \tPnew]^A_{\tB} \xlra\cong [X, \Pnew]^A_B.\]

We will now apply this to a special factorization of $g \colon X \to B$ with the first map the identity, so that $\tB=X$, and the induced pullback square:
\[\xymatrix{
A \ar[r]_-\tfn \ar@/^1pc/[rr]^-{\fn} \ar@{ >->}[d]_-\iota & \tPnew \ar[r] \ar@{->>}[d]_-\tpsin \pb & \Pnew \ar@{->>}[d]^-\psin \\
X \ar[r]^-\id \ar@/_1pc/[rr]_-g & X \ar[r] & B
}\]
Viewing $\tPnew = X \times_B \Pnew$ as a subspace of $X \times \Pnew$, the map $\tpsin$ becomes the projection onto the left factor and $\tfn = (\iota, \fn)$. A diagonal $\onew \colon X \to \Pnew$ in the composite square then induces a diagonal $\tonew = (\id, \onew) \colon X \to \tPnew$ in the left square and this is simply a section of $\tpsin$. In addition, it is easy to verify that $\tfn = \tonew \iota$, i.e.\ $\tfn$ takes values on this section.

\begin{proposition} \label{prop:homotopy_classes_pullback}
There is a bijection $[X, \Pnew]^A_B \cong [X, \tPnew]^A_\tB$, computable in both directions, with the latter lifting-extension problem pointed.
\end{proposition}

\begin{remark}
There is a different factorization $g \colon X \xra{\onew} \Pnew \xra{\psin} B$ and the induced pullback $\Pnew \times_B \Pnew \to \Pnew$ also admits a section by the diagonal map. The advantage of this pullback is that it depends only on $\psi \colon Y \to B$ and not on $A$, $X$, $f$ or $g$. On the other hand, it seems bigger than the pullback $\tPnew$ proposed above. An \Hopf space structure on $\Pnew \times_B \Pnew$ can be interpreted directly as a structure on $\Pnew$, given by a ternary operation and related to heaps; this approach has been developed in \cite{heaps2}.
\end{remark}

\addtocounter{equation}{-1}
{\renewcommand{\theequation}{\ref{t:semi_eff} (restatement)}
\begin{theorem}
\se
\end{theorem}}

\begin{proof}
This is a corollary of a collection of results obtained so far. By Proposition~\ref{prop:homotopy_classes_pullback}, we may replace the Moore--Postnikov tower and the given lifting-extension problem by their pointed versions. By Proposition~\ref{prop:weak_H_space_structure}, it is possible to construct on $\tPnew$ the structure of a weak \Hopf space. By results of this subsection, it is possible to strictify this structure, making $\tPnew$ into an \Hopf space. According to Theorem~\ref{t:existence_of_H_space_structures}, it is homotopy associative, homotopy commutative and with a right homotopy inverse; consequently, $[X,\Pnew]^A_B \cong [X,\tPnew]^A_\tB$ is an abelian group. By Proposition~\ref{prop:addition_on_homotopy_classes}, it is possible to compute the addition and the inverse in the latter group on the level of representatives, making it into a semi-effective abelian group. Since the isomorphism is computable in both directions by Proposition~\ref{prop:weak_H_space_structure}, $[X,\Pnew]^A_B$ also becomes a semi-effective abelian group.
\end{proof}

\heading{Proof of Theorems~\ref{thm:exact_sequence_short} and~\ref{thm:exact_sequence_long}} \label{sec:proof_of_exact_sequence}

We will need the fact that $\pn \colon \Pnew \to \Pold$ is a principal fibration with fibrewise action of $\Ln=B\times K(\pin,\thedim)$ (in fibrewise world, an action is a map $\Pnew \times_B \Ln \to \Pnew$). Thinking of $\Pnew$ as a subset of $\Pold\times\En$, an element $z\in\Ln$ acts on $(x,c)\in\Pnew$ by $(x,c)+z\defeq(x,c+z)$ where the sum $c+z$ is taken within the fibrewise simplicial group $\En$.

\addtocounter{equation}{-1}
{\renewcommand{\theequation}{\ref{thm:exact_sequence_short} (restatement)}
\begin{theorem}
\exseqshort
\end{theorem}}

\begin{proof}
Both claims will be proved as a part of the proof of the following theorem.
\end{proof}

\addtocounter{equation}{-1}
{\renewcommand{\theequation}{\ref{thm:exact_sequence_long} (restatement)}
\begin{theorem}
\exseqlong
\end{theorem}}

\begin{proof}
We start by defining the map $\jnst$: on the level of maps, $\jnst(\zeta) = \onew + \zeta$ (the action of $\Ln$ on $\Pnew$) and it passes to homotopy classes. We then obtain a sequence
\[[X,\Ln]^A_B \xlra{\jnst} [X,\Pnew]^A_B \xlra{\pnst} [X,\Pold]^A_B \xlra{\knst} [X,\Kn]^A_B,\]
whose exactness at the second term is simple. To prove exactness at the third term, we recall that $0 \in [X, \Kn]^A_B$ is the only element in the image of $\delta_*$, as remarked after Lemma~\ref{l:fully_eff_cohlgy_snd}. Thus, $[\ell_{n-1}] \in \im \pnst$ iff $\ell_{n-1}$ lifts to $\Pnew$ iff $\kn\ell_{n-1}$ lifts to $\En$ iff $[\kn\ell_{n-1}] \in \im \delta_*$ iff $\knst[\ell_{n-1}] = 0$.

It is possible to extend the sequence to the left by $[\stdsimp{1}\times X,P]^{(\partial\stdsimp{1}\times X)\cup(\stdsimp{1}\times A)}_B$, the set of homotopy classes of fibrewise homotopies $\stdsimp 1 \times X \to \Pold$ from $\oold$ to $\oold$ relative to $A$; its base point is the constant homotopy at $\oold$. First, we describe
\[\connn \colon [\stdsimp{1}\times X,\Pold]^{(\partial\stdsimp{1}\times X)\cup(\stdsimp{1}\times A)}_B\ra[X,\Ln]^A_B.\]
Let $h\col \stdsimp{1}\times X\ra\Pold$ be a homotopy as prescribed above. Choose a lift $\widetilde h$ of $h$ along $\pn\col\Pnew\ra\Pold$ that starts at $\onew$ and is relative to $A$ (this can be carried out in an algorithmic way by Proposition~\ref{prop:homotopy_lifting}). Restricting to the end of the homotopy prescribes a map $\widetilde h_\mathrm{end}\col X\ra\Pnew$ that lies over $\oold$ and is thus of the form $\widetilde h_\mathrm{end} = \onew + \zeta$ for a unique map $\zeta \colon X\ra\Ln$. We set $\connn[h] = [\zeta]$; this is well defined by Lemma~\ref{l:connecting_homomorphism}.

By definition, $\jnst\connn[h] = [\widetilde h_\mathrm{end}]$ and the homotopy $\widetilde h$ shows this equal to $[\onew]$. The exactness is also easy -- a homotopy $\widetilde h \colon \onew \sim \onew + \zeta$ in $\Pnew$ projects down to $\Pold$ to a homotopy $\stdsimp{1}\times X\ra\Pold$ representing a preimage of $[\zeta]$ (since its lift is $\widetilde h$ with the appropriate end $\onew + \zeta$). To summarize, we have an exact sequence of pointed sets
\[[\stdsimp1\times X,\Pold]^{(\partial\stdsimp{1}\times X)\cup(\stdsimp{1}\times A)}_B\ra[X,\Ln]^A_B\ra[X,\Pnew]^A_B\ra[X,\Pold]^A_B\ra[X,\Kn]^A_B.\]

Our next aim is to show that the maps in this sequence are homomorphisms of groups. By replacing the stages $\Pold$, $\Pnew$ by their pullbacks $\tPold$, $\tPnew$ if necessary, we may assume that these stages are pointed and that so is the lifting-extension problem in question. Therefore, the addition on homotopy classes is defined through strict \Hopf space structures (namely, strictifications of weak \Hopf space structures), as described in Section~\ref{sec:strictness}. According to the uniqueness part of Theorem~\ref{t:existence_of_H_space_structures}, we may assume that the strict \Hopf space structure is constructed as in Section~\ref{sec:non_constr_H_space}. The corresponding non-additivity map $m'$ and its lift $M'$ are as in the following diagram
\[\xymatrix{
\Pold\vee_B\Pold \ar[r]^-0 \ar[d]_-\vartheta & \En \ar[d]^-\delta \\
\Pold\times_B\Pold \ar[r]_-{m'} \ar@{-->}[ru]^-{M'} & \Kn
}\]
and the addition is defined on $\Pnew$ inductively using
\[(x,y)+'(x',y')=(x+'x',y+_{\qno}y'+M'(x,x')).\]
An important property is that this makes $\jn$ into a homomorphism (since $M'$ vanishes when one of the arguments lies on the zero section $\oold$). Therefore, already on the level of representatives, $\pnst$ and $\jnst$ are homomorphisms.

Since $\knst$ preserves zeros, it is enough to show that $\knst$ is an affine homomorphism. Applying \eqref{eq:affine_homomorphism} to $\knst$, we need a homotopy
\[\kn x+_{\kno}\kn y\sim\kn(x+'y)\]
relative to $\Pold \vee_B \Pold$ or, in other words, a relative homotopy $0 \sim m'$. Such a homotopy is obtained as an image under $\delta$ of a relative homotopy $0 \sim M'$, which exists since $\En$ is (fibrewise) contractible.

It remains to treat the connecting homomorphism $\connn$. If $h_0,h_1\col \stdsimp{1}\times X\ra\Pold$ represent two elements of the domain, then the lift of $h_0\hplus h_1$ may be chosen to be the sum $\widetilde h_0\hplus\widetilde h_1$ of the two lifts. Thus, $(\widetilde h_0\hplus\widetilde h_1)_\mathrm{end}=(\widetilde h_0)_\mathrm{end}\hplus(\widetilde h_1)_\mathrm{end}$ and this corresponds to the sum of the $\connn$-images.

\subsection*{Computability of sections}
A section of $\pnst$ is defined by mapping a partial diagonal $\ell\col X\ra\Pold$ to an arbitrary lift $\widetilde\ell\col X\ra\Pnew$ of $\ell$, with a prescribed restriction to $A$. The computation of $\widetilde\ell$ is taken care of by Proposition~\ref{prop:lift_ext_one_stage}; a lift exists because $\ker\knst=\im\pnst$.

For the construction of a section $\sigma$ of $\jnst$, let $\ell\col X\ra\Pnew$ be a diagonal such that $\pn\ell$ is homotopic to $\oold$. Proposition~\ref{prop:compute_homotopy} computes a homotopy $\oold \sim \pn\ell$. Using Proposition~\ref{prop:homotopy_lifting}, we lift it along $\pn$ to a homotopy $\ell' \sim \ell$, relative to $A$, for some $\ell'$. Since $\pn\ell' = \oold = \pn\onew$, we have $\ell' = \onew + \zeta$ for a unique $\zeta \colon X \to \Ln$ and we set $\sigma(\ell)=\zeta$.
\end{proof}

\begin{lemma} \label{l:connecting_homomorphism}
Continuing the notation from the proof of Theorem~\ref{thm:exact_sequence_long}, the homotopy class $[\zeta]$ does not depend on the choices made; thus, $\connn$ is a well defined map. In addition, if $\zeta'$ is any other representative of this homotopy class, i.e.\ $\connn[h] = [\zeta']$, there exists a lift $\widetilde h'$ of $h$ that is a homotopy between $\onew$ and $\onew + \zeta'$ relative to $A$.
\end{lemma}

\begin{proof}
If $h$ is homotopic to $h'$, by a homotopy relative to $(\partial \stdsimp 1 \times X) \cup (\stdsimp 1 \times A)$, and $\widetilde h'$ is any lift of $h'$, then we may lift the homotopy $h \sim h'$ to a homotopy $\widetilde h \sim \widetilde h'$ relative to $(0 \times X) \cup (\stdsimp 1 \times A)$,\footnote{%
	This is a solution of a lifting extension problem whose left part is an inclusion in the pair $(\stdsimp 1, \partial \stdsimp 1) \times (\stdsimp 1, 0) \times (X, A)$ with the middle term $\infty$-connected, thus also the whole product, and the inclusion is a weak homotopy equivalence.
} that restricts to $1 \times X$ to a fibrewise homotopy $\onew + \zeta \sim \onew + \zeta'$, relative to $A$, implying $\zeta \sim \zeta'$; thus, $\connn$ is well defined.

For the second part, concatenating the homotopy $\widetilde h \colon \onew \sim \onew + \zeta$, with the homotopy $\onew + \zeta \sim \onew + \zeta'$ induced from the given $\zeta \sim \zeta'$, we obtain $\widetilde h' \colon \onew \sim \onew + \zeta'$. If the concatenation of homotopies is computed, as in Proposition~\ref{prop:homotopy_concatenation}, using the lift in
\[\xymatrix{
(\horn{2}{1} \times X) \cup (\Delta^2 \times A) \ar[rr] \ar@{ >->}[d] & & \Pnew \ar@{->>}[d]^-\pn \\
\stdsimp 2 \times X \ar[r]_-{s^1\times\id} \ar@{-->}[rru] & \stdsimp 1 \times X \ar[r]_-{h} & \Pold
}\]
then this concatenation will also be a lift of $h$, since the restriction of $h(s^1 \times \id)$ to $d_1 \stdsimp 2 \times X$ equals $h$; here $s^1 \colon \Delta^2 \to \Delta^1$ is the map sending the non-degenerate 2-simplex of $\Delta^2$ to the $s_1$-degeneracy of the non-degenerate 1-simplex of $\Delta^1$.
\end{proof}

\begin{proposition}\label{prop:compute_homotopy}
Suppose that $(X, A)$ is equipped with effective homology and that $[\stdsimp 1 \times X, \Polder]^{(\partial\stdsimp 1 \times X) \cup (\stdsimp 1 \times A)}_B$ is a fully effective abelian group for all $\thedimm < \thedim - 1$. Then there is an algorithm that decides whether given $[\oold], [\ell_{\thedim-1}] \in [X, \Pold]^A_B$ are equal. If this is the case, the algorithm computes a homotopy $\oold \sim \ell_{\thedim-1}$.
\end{proposition}

We remark that the above homotopy decision algorithm admits a generalization to non-stable stages and, thus, provides homotopy testing for maps to an arbitrary simply connected space, see~\cite{Filakovsky:suspension}.

\begin{proof}
We compute the homotopy $h_{\thedim-1}$ by induction on the height $\thedimm$ of the Moore--Postnikov stage $\Polder$. Let $o_\thedimm$ and $\ell_\thedimm$ denote the projections of $\oold$ and $\ell_{\thedim-1}$ onto the $\thedimm$-th stage $P_\thedimm$. Suppose that we have computed a homotopy $h_{\thedimm-1} \colon o_{\thedimm-1} \sim \ell_{\thedimm-1}$ and lift it by Proposition~\ref{prop:homotopy_lifting} to a homotopy $\widetilde h_{\thedimm-1}\col\ell'_\thedimm\sim\ell_\thedimm$ from some map $\ell'_\thedimm$, necessarily of the form $\ell'_\thedimm = o_\thedimm + \zeta'_\thedimm$.

Since Proposition~\ref{prop:homotopy_concatenation} provides algorithmic means for concatenating homotopies, it remains to construct a homotopy $h'_\thedimm \colon o_\thedimm \sim \ell'_\thedimm$. Consider the connecting homomorphism in \eqref{eq:les} for stages $P_{\thedimm-1}$ and $P_\thedimm$, i.e.
\[\conni\col[\stdsimp{1}\times X,P_{\thedimm-1}]^{(\partial\stdsimp{1}\times X)\cup(\stdsimp{1}\times A)}_B\lra[X,\Li]_B^A.\]
From the already proved exactness of \eqref{eq:les} and from $\ell'_\thedimm\sim\ell_\thedimm$, it follows that $[\zeta'_\thedimm]$ lies in the image of $\conni$ if and only if $o_\thedimm \sim \ell_\thedimm $. If this is the case, we obtain a representative $h'_{\thedimm-1}$ of a preimage by Lemma~\ref{l:preimage}. Thus, $\conni[h'_{\thedimm-1}]=[\zeta'_\thedimm]$.

According to Lemma~\ref{l:connecting_homomorphism}, there exists a lift of $h'_{\thedimm-1}$ that is a homotopy $h'_{\thedimm} \colon o_{\thedimm} \sim o_{\thedimm} + \zeta'_\thedimm$ relative to $A$.  This specifies the top map in the following lifting-extension problem
\[\xymatrix{
(\partial\stdsimp{1}\times X)\cup(\stdsimp{1}\times A) \ar[r] \ar[d] & P_\thedimm \ar[d] \\
\stdsimp{1}\times X \ar[r]_-{h'_{\thedimm-1}} \ar@{-->}[ru]_-{h'_\thedimm} & P_{\thedimm-1}
}\]
and $h'_{\thedimm}$ can thus be computed using Proposition~\ref{prop:lift_ext_one_stage}.
\end{proof}

\section{Leftover proofs} \label{sec:leftover_proofs}

The purpose of this section is to prove statements that were used in the main part but whose proofs would disturb the flow of the paper.

\addtocounter{equation}{-1}
{\renewcommand{\theequation}{\ref{t:MP_tower} (restatement)}
\begin{theorem}
\MPt
\end{theorem}}\label{sec:MP_tower_proof}

The proof will be presented in two parts. First, we describe the construction of the objects and then we prove that they really constitute an extended Moore--Postnikov tower.

The construction itself follows ideas by E.~H.~Brown for non-equivariant simplicial sets in \cite{Brown} and by C.~A.~Robinson for topological spaces with free actions of a group in \cite{Robinson}.

We described the construction in the non-equivariant non-fibrewise case $G=1$ and $B=*$ in detail in \cite{polypost}. Here, we give a brief overview with the emphasis on the necessary changes for $G$ and $B$ non-trivial.

\subsection*{Construction}
The first step of the construction is  easy. Put $P_0=B$ and $\varphi_0=\psi$. To proceed by induction, suppose that we have constructed $\Pold$ and a map $\varphi_{\thedim-1}\col Y\to \Pold$ with properties~\ref{MP1} and~\ref{MP2} from the definition of the Moore--Postnikov tower. Moreover, assume that $\Pold$ is equipped
with effective homology.

Viewing $\cone\varphi_{(\thedim-1)*}$ as a perturbation of $C_*\Pold\oplus C_*Y$, we obtain from strong equivalences $C_*\Pold\LRa C_*^\ef\Pold$ and $C_*Y\LRa C_*^\ef Y$ a strong equivalence $\cone\varphi_{(\thedim-1)*}\LRa C_*^\ef$ with $C_*^\ef$ effective (for details, see \cite[Proposition~3.8]{polypost}). Let us consider the composition
\[C_{\thedim+1}^\ef\ra Z_{\thedim+1}(C_*^\ef)\ra H_{\thedim+1}(C_*^\ef)\defeq\pin,\]
where the first map is an (equivariant) retraction of $Z_{\thedim+1}(C_*^\ef)\subseteq C_{\thedim+1}^\ef$, computed by the algorithm of Proposition~\ref{prop:projectivity}; the second map is simply the projection onto the homology group. The homology group itself is computed from $C_*^\ef$ -- by forgetting the action of $G$, it is a chain complex of finitely generated abelian groups and Smith normal form is available. The $G$-action on $\pin$ is easily computed from the $G$-action on $C_*^\ef$. Composing with the chain map $\cone\varphi_{(\thedim-1)*}\ra C_*^\ef$ coming from the strong equivalence, we obtain
\[\kappa+\lambda\col C_{\thedim+1}\Pold\oplus C_\thedim Y=\big(\cone\varphi_{(\thedim-1)*}\big)_{\thedim+1}\ra C_{\thedim+1}^\ef\ra\pin\]
whose components are denoted $\kappa$ and $\lambda$. They correspond, respectively, to maps
\[\knp\col \Pold\ra K(\pin,\thedim+1),\ l_\thedim'\col Y\ra E(\pin,\thedim)\]
that fit into a square
\begin{equation}\label{eq:postnikov_square}
\xymatrixc{
Y \ar[r]^-{l_\thedim'} \ar[d]_-{\varphi_{\thedim-1}} & E(\pin,\thedim) \ar[d]^-\delta \\
\Pold \ar[r]_-{\knp} & K(\pin,\thedim+1)
}\end{equation}
which commutes by the argument of \cite[Section~4.3]{polypost}.

Now we can take $\Pnew=\Pold\times_{K(\pin,\thedim+1)}E(\pin,\thedim)$ to be the pullback as in part~\ref{MP3} of the definition of the tower. By the commutativity of the square \eqref{eq:postnikov_square}, we obtain a map $\varphin=(\varphi_{\thedim-1},l_\thedim')\col Y\ra \Pnew$ as in
\[\xymatrix@=10pt{
Y \ar@/^10pt/[drrr]^{l_\thedim'} \ar@/_10pt/[dddr]_{\varphi_{\thedim-1}} \ar@{-->}[dr]^-{\varphin} \\
& \Pnew \ar[rr] \ar[dd]^{\pn} & & E(\pin,\thedim) \ar[dd]^{\delta} \\
\\
& \Pold \ar[rr]_-{\knp} & & K(\pin,\thedim+1)
}\]
which we will prove to satisfiy the remaining conditions for the $\thedim$-th stage of a Moore--Postnikov tower.

First, however, we equip $\Pnew$ with effective homology. To this end, observe that $\Pnew$ is isomorphic to the twisted cartesian product $\Pold\times_{\tau}K(\pin,\thedim)$, see \cite[Proposition~18.7]{May:SimplicialObjects-1992}. Since $\Pold$ is equipped with effective homology by induction, and $K(\pin,\thedim)$ admits effective homology non-equivariantly by \cite[Theorem~3.16]{polypost}, it follows from \cite[Corollary~12]{Filakovsky} (or \cite[Proposition~3.10]{polypost}) that $\Pnew$ can also be equipped with effective homology non-equivariantly. Since the $G$-action on $\Pnew$ is clearly free (any fixed point would get mapped by $\psin$ to a fixed point in $B$), Theorem~\ref{t:vokrinek} provides (equivariant) effective homology for $\Pnew$ (distinguished simplices of $\Pnew$ are pairs with the component in $\Pold$ distinguished).

\subsection*{Correctness}
From the exact sequence of homotopy groups associated with the fibration sequence
\[\Pnew\ra \Pold\ra K(\pin,\thedim+1)\]
and the properties~\ref{MP1} and~\ref{MP2} for $\Pold$, we easily get that $\Pnew$ satisfies the condition~\ref{MP2} and that $\varphinst\col \pi_\thedimm(Y)\to\pi_\thedimm(\Pnew)$ is an isomorphism for $0\le\thedimm\le\thedim-1$.

The rest of the proof is derived, as in \cite[Section~4.3]{polypost}, from the morphism of long exact sequences of homotopy groups
\[\xymatrix@R=15pt@C=10pt{
& \pi_{\thedim+1}(Y)\ar[r] \ar[d]^{\varphinst}& \pi_{\thedim+1}(\cyl\varphi_{\thedim-1})\ar[r] \ar[d]^{\cong} & \pi_{\thedim+1}(\cyl\varphi_{\thedim-1},Y) \ar[d]^{\cong} \ar[r] & \pin(Y) \ar[d]^{\varphinst} \ar[r] & \pin(\cyl\varphi_{\thedim-1}) \ar[d]^{\cong} \ar[r] & 0 \\
0 \ar[r] & \pi_{\thedim+1}(\Pnew)\ar[r] & \pi_{\thedim+1}(\cyl \pn) \ar[r] & \pi_{\thedim+1}(\cyl \pn,\Pnew) \ar[r] & \pin(\Pnew) \ar[r] & \pin(\cyl \pn)
}\]
associated with pairs $(\cyl\varphi_{\thedim-1},Y)$ and $(\cyl \pn,\Pnew)$. The arrow in the middle is an isomorphism by \cite[Lemma~4.5]{polypost}, while the remaining two isomorphisms are consequences of the fact that both cylinders deform onto the same base $\Pold$. The zero on the left follows from the fact that the fibre of $\pn$ is $K(\pin,\thedim)$ and the zero on the right comes from the condition~\ref{MP1} for $\Pold$. By the five lemma, $\varphinst$ is an isomorphism on $\pin$ and an epimorphism on $\pi_{\thedim+1}$ which completes the proof of condition~\ref{MP1}.

\subsection*{Addendum}
For a given computable $\beta \colon \tB \to B$, the pullbacks $\tPnew = \tB \times_B \Pnew$ may be identified with twisted cartesian products $\tPold \times_\tau K(\pin, \thedim)$ and as such admit effective homology by induction, starting from the assumed effective homology of $\widetilde P_0 = \tB$.
\qed
\vskip\topsep

For the next proof, we will use the following observation.

\begin{lemma}\label{l:fibrant_replacement}
Every map $\psi\col P\to Q$ can be factored as $\psi\col P\cof[j]P'\fib[\psi']Q$, where $j$ is a weak homotopy equivalence and $\psi'$ is a Kan fibration.
\end{lemma}

By a \emph{weak homotopy equivalence}, we will understand a map whose geometric realization is a $G$-homotopy equivalence.

\begin{proof}
This is the small object argument (see e.g.\ \cite[Section~10.5]{Hirschhorn} or~\cite[Section~7.12]{DwyerSpalinski}) applied to the collection $\mcJ$ of ``$G$-free horn inclusions'' $G\times\horn\thedim\thedimm\to G\times\stdsimp\thedim$, $\thedim\geq1$, $0\leq\thedimm\leq\thedim$. Using the terminology of \cite{Hirschhorn}, the $\mcJ$-injectives are exactly those maps that have non-equivariantly the right lifting property with respect to $\horn\thedim\thedimm\to \stdsimp\thedim$ (this follows from the equivalence \eqref{eq:horn-injectivity} from the next proof), i.e.\ Kan fibrations. The geometric realization of every relative $\mcJ$-cell complex is a $G$-homotopy equivalence since the geometric realization of $G\times\stdsimp\thedim$ clearly deforms onto that of $G\times\horn\thedim\thedimm$.
\end{proof}

\addtocounter{equation}{-1}
{\renewcommand{\theequation}{\ref{t:n_equivalence} (restatement)}
\begin{theorem}
\nequiv
\end{theorem}}\label{sec:n_equivalence_proof}

\begin{proof}
By construction, $\varphin\col Y\to \Pnew$ is an $(\thedim+1)$-equivalence. By the proof of Lemma~\ref{l:fibrant_replacement}, we may assume $Y\to Y'$ to be a relative $\mcJ$-cell complex. We show that $\varphin$ factors through $Y'$. Given that this is true for $\Pold$, we form the square
\[\xymatrix@C=40pt{
Y \ar[r]^-{\varphin} \ar@{ >->}[d]_-\sim & \Pnew \ar@{->>}[d]^\pn \\
Y' \ar[r]_-{\varphi_{\thedim-1}'} \ar@{-->}[ur]^-{\varphinp} & \Pold
}\]
in which a diagonal exists by the fact that $Y\to Y'$ is a relative $\mcJ$-cell complex and such maps have the left lifting property with respect to Kan fibrations.

The map $\varphinp$ is also an $(\thedim+1)$-equivalence. We will prove more generally that
\[\psi_*\col[X,P]^A_B\to[X,Q]^A_B\]
is an isomorphism for any $(\thedim+1)$-equivalence $\psi\col P\to Q$.

The basic idea is that $X$ is built from $A$ by consecutively attaching ``cells with a free action of $G$'', namely $X=\cup X_\thedimm$ and in each step $X_\thedimm=X_{\thedimm-1}\cup_{G\times\partial\stdsimp{\theotherdim_\thedimm}}G\times\stdsimp{\theotherdim_\thedimm}$ with $\theotherdim_\thedimm\leq\thedim$.\footnote{Thus, the action needs only be free away from $A$ and the same generalization applies to the dimension.}

First, we prove that $\psi_*$ is surjective under the assumption that $\psi$ is an $\thedim$-equivalence. For convenience, we replace $\psi$ by a $G$-homotopy equivalent Kan fibration using Lemma~\ref{l:fibrant_replacement}. Suppose that the above map $\psi_*$, but with $X$ replaced by $X_{\thedimm-1}$, is surjective and we prove the same for $X_\thedimm$. This is clearly implied by the solvability of the following lifting-extension problem
\[\xymatrix@C=40pt{
X_{\thedimm-1} \ar[r] \ar@{ >->}[d] & P \ar@{->>}[d]^-\psi \\
X_\thedimm \ar[r]_-{\ell} \ar@{-->}[ru] & Q
}\]
(to find a preimage of $[\ell]$ at the bottom, we find the top map by the inductive hypothesis; if the lift exists, it gives a preimage of $[\ell]$ as required). As $X_\thedimm$ is obtained from $X_{\thedimm-1}$ by attaching a single cell, the problem is equivalent to
\begin{equation}\label{eq:horn-injectivity}
\xy *!C\xybox{\xymatrix@C=30pt{
G\times\partial\stdsimp{\theotherdim_\thedimm} \ar[r] \ar@{ >->}[d] & P \ar@{->>}[d]^-\psi \ar@{}[drrr]|-{\parbox{\widthof{that is further}}{that is further equivalent to}} & & & \partial\stdsimp{\theotherdim_\thedimm} \ar[r] \ar@{ >->}[d] & P \ar@{->>}[d]^-\psi \\
G\times\stdsimp{\theotherdim_\thedimm} \ar[r] \ar@{-->}[ru] & Q & & & \stdsimp{\theotherdim_\thedimm} \ar[r] \ar@{-->}[ru] & Q
}}\endxy
\end{equation}
where the problem on the right is obtained from the left by restricting to $e\times\stdsimp{\theotherdim_\thedimm}$ and is non-equivariant. Its solution is guaranteed by $\psi$ being an $\theotherdim_\thedimm$-equivalence.

To prove the injectivity of $\psi_*$, we put back the assumption of $\psi$ being an $(\thedim+1)$-equivalence. We study the preimages of $[\ell]\in[X,Q]^A_B$ under $\psi_*$; these clearly form $[X,P]^A_Q$. By the surjectivity part, this set is non-empty. By pulling back $P$ along $\ell$, we thus obtain a fibration $\ell^*P\to X$ with a section $X\to\ell^*P$ which is an $\thedim$-equivalence.\footnote{%
	The fibres of $\psi$ are $\thedim$-connected and isomorphic to those of $\ell^*P\to X$. From the long exact sequence of homotopy groups of this fibration, it follows that $\ell^*P\to X$ is also an $(\thedim+1)$-equivalence and its section then must be an $\thedim$-equivalence.
} Thus,
\[[X,P]^A_Q\cong[X,\ell^*P]^A_X\xlla\cong[X,X]^A_X=*\]
by the surjectivity part (any surjection from a one-element set is a bijection).
\end{proof}

Next, we need the following lemma.

\begin{lemma}\label{l:whe_weak_strict}
The natural maps $P\hvee_BP\ra P\vee_BP$ and $P\htimes_BP\ra P\times_BP$ are weak homotopy equivalences.

The inclusion $(P\htimes_BP)\cup\big(\bdry_2\stdsimp{2}\times(P\vee_BP)\big)\cof\stdsimp2\times(P\times_BP)$ is a weak homotopy equivalence.
\end{lemma}

\begin{proof}
The space $P\hvee_BP$ is naturally a subspace of $\bdry_2\stdsimp2\times(P\vee_BP)$ and it is enough to show that it is in fact a deformation retract. A continuous deformation is obtained from a deformation of $\bdry_2\stdsimp2\times P_\mathrm{right}$ onto $(0\times P_\mathrm{right})\cup(\bdry_2\stdsimp2\times B)$ and a symmetric deformation of $\bdry_2\stdsimp2\times P_\mathrm{left}$ onto $(1\times P_\mathrm{left})\cup(\bdry_2\stdsimp2\times B)$.

To prove the remaining claims, consider the deformation of $\stdsimp2\times(P\times_BP)$ onto $2\times(P\times_BP)$, given by deforming $\stdsimp2$ linearly onto $2$ and by a constant homotopy at identity on the second component $P\times_BP$. By an easy inspection, it restricts to a deformation of $P\htimes_BP$ onto $2\times(P\times_BP)$, giving the second claim.

Since both $\stdsimp2\times(P\times_BP)$, $P\htimes_BP$ deform onto the same $2\times(P\times_BP)$, it is enough for the last claim to find a deformation of
\[(P\htimes_BP)\cup\big(\bdry_2\stdsimp{2}\times(P\vee_BP)\big)\]
onto $P\htimes_BP$. This is provided by the deformation of $\bdry_2\stdsimp2\times(P\vee_BP)$ onto $P\hvee_BP$ (the intersection of the two spaces in the union above) from the first paragraph.
\end{proof}

Now we are ready to prove the following proposition.

\addtocounter{equation}{-1}
{\renewcommand{\theequation}{\ref{prop:weak_H_space_connect} (restatement)}
\begin{proposition}
\wHsc
\end{proposition}}

\begin{proof}
By the first part of the previous lemma, we may replace the pair in the statement by $(\Pnew\times_B\Pnew,\Pnew\vee_B\Pnew)$.

First, we recall that $\Pnew\ra B$ is a minimal fibration (each $\delta\col E(\pi_\thedimm,\thedimm)\ra K(\pi_\thedimm,\thedimm+1)$ is one and the class of minimal fibrations is closed under pullbacks and compositions, see \cite{May:SimplicialObjects-1992}). It is well known that over each simplex $\sigma\col \stdsimp{\thedimm}\ra B$ any minimal fibration is trivial and it is easy to modify this to an isomorphism $\sigma^*\Pnew\cong\stdsimp{\thedimm}\times F$ of fibrations with sections, where $F$ denotes the fibre of $\Pnew\ra B$ and is $\theconn$-connected by the assumptions.\footnote{%
	Start with an inclusion $(\stdsimp{\thedimm}\times*)\cup(0\times F)\ra\sigma^*\Pnew$ given by the zero section on the first summand and by the inclusion on the second. Extend this to a fibrewise map $\stdsimp{\thedimm}\times F\ra\sigma^*\Pnew$ which is a fibrewise homotopy equivalence, hence an isomorphism, by the minimality of $\Pnew\ra B$.
}. Consequently, $\Pnew\vee_B\Pnew$ is a fibre bundle with fibre $F\vee F$. Thus, we have a map of fibre sequences
\[\xymatrix{
F\vee F \ar[r] \ar[d] & \Pnew\vee_B\Pnew \ar[r] \ar[d] & B \ar@{=}[d] \\
F\times F \ar[r] & \Pnew\times_B\Pnew \ar[r] & B
}\]
The left map is $(2\theconn+1)$-connected. By the five lemma applied to the long exact sequences of homotopy groups, the middle map $\Pnew\vee_B\Pnew\ra \Pnew\times_B\Pnew$ is also $(2\theconn+1)$-connected.

To show that the equivariant cohomology groups vanish, we make use of a contraction of $C_*(\Pnew\htimes_B\Pnew)$ onto $C_*(\Pnew\hvee_B\Pnew)$ in dimensions $\leq2\theconn+1$; its existence follows from the proof of Proposition~\ref{prop:projectivity}. By the additivity of $\operatorname{Hom}_{\ZG}(-,\pi)$, there is an induced contraction of $C^*_G(\Pnew\htimes_B\Pnew;\pi)$ onto $C^*_G(\Pnew\hvee_B\Pnew;\pi)$ and thus the relative cochain complex is acyclic.
\end{proof}

\addtocounter{equation}{-1}
{\renewcommand{\theequation}{\ref{t:existence_of_H_space_structures} (restatement)}
\begin{theorem}
\eoHss
\end{theorem}}\label{sec:existence_of_H_space_structures_proof}

\begin{proof}
By the previous proposition, the left vertical map in
\[\xymatrix@C=40pt{
\Pnew\vee_B\Pnew \ar[r]^-\nabla \ar@{ >->}[d]_\vartheta & \Pnew \ar@{->>}[d]^-{\psin} \\
\Pnew\times_B\Pnew \ar[r] \ar@{-->}[ru]_-\add & B
}\]
is $(2\theconn+1)$-connected. Since the homotopy groups of the fibre of $\psin$ are concentrated in dimensions $\theconn\leq\thedimm\leq\thedim$, the relevant obstructions (they can be extracted from the proof of Proposition~\ref{prop:lift_ext_one_stage}) for the existence of the diagonal lie in
\[H^{\thedimm+1}_G(\Pnew\times_B\Pnew,\Pnew\vee_B\Pnew)=0\]
(since $i+1\leq\thedim+1\leq 2\theconn+1$). The diagonal is unique up to homotopy by the very same computation. Thus, in particular, replacing $\add$ by the opposite addition $\add^\op\col(x,y)\mapsto y+x$ yields a homotopic map, proving homotopy commutativity. Similarly, homotopy associativity follows from the uniqueness of a diagonal in
\[\xymatrix@C=40pt{
(B\times_B\Pnew\times_B\Pnew)\cup(\Pnew\times_B\Pnew\times_BB) \ar[r] \ar@{ >->}[d] & \Pnew \ar@{->>}[d]^-{\psin} \\
\Pnew\times_B\Pnew\times_B\Pnew \ar[r] \ar@{-->}[ru] & B
}\]
(the pair on the left is again $(2\theconn+1)$-connected) with two diagonals specified by mapping $(x,y,z)$ to $(x+y)+z$ and $x+(y+z)$.

The existence of a homotopy inverse is a fibrewise and equivariant version of \cite[Theorem~3.4]{Stasheff}; the proof applies without any complications when the action of $G$ is free. We will not provide more details since we construct the inverse directly in Section~\ref{sec:H_space_constr}.
\end{proof}

For the next proof, we will use a general lemma about filtered chain complexes. Let $C_*$ be a chain complex equipped with a filtration
\[0 = F_{-1}C_* \subseteq F_0C_* \subseteq F_1C_* \subseteq \cdots\]
such that $C_* = \bigcup_i F_iC_* $. As usual, we assume that each $F_iC_*$ is a $\ZG$-cellular subcomplex, i.e.\ generated by a subset of the given basis of $C_*$. We assume that this filtration is \emph{locally finite}, i.e.\ for each $n$, we have $C_n = F_iC_n$ for some $i \geq 0$. For the relative version, let $D_*$ be a ($\ZG$-cellular) subcomplex of $C_*$ and define $F_iD_* = D_* \cap F_iC_*$.

\begin{lemma}\label{l:filt}
Under the above assumptions, if each filtration quotient $G_iC_* = F_iC_*/F_{i-1}C_*$ has effective homology then so does $C_*$. More generally, if each $(G_iC_*,G_iD_*)$ has effective homology then so does $(C_*,D_*)$.
\end{lemma}

\begin{proof}
We define $G_* = \bigoplus_{i \geq 0} G_iC_*$, the associated graded chain complex. Then $C_*$ is obtained from $G_*$ via a perturbation that decreases the filtration degree $i$. Taking a direct sum of the given strong equivalences $G_iC_* \La \widehat G_iC_* \Ra G_i^\ef C_*$, we obtain a strong equivalence $G_* \La \widehat G_* \Ra G_*^\ef$ with all the involved chain complexes equipped with a ``filtration'' degree. Since the perturbation on $G_*$ decreases this degree, while the homotopy operator preserves it, we may apply the perturbation lemmas, Propositions~\ref{p:epl} and~\ref{p:bpl}, to obtain a strong equivalence $C_* \La \widehat C_* \Ra C_*^\ef$.
\end{proof}

\addtocounter{equation}{-1}
{\renewcommand{\theequation}{\ref{p:effective_homotopy_colimits} (restatement)}
\begin{proposition}
\ehc
\end{proposition}}\label{sec:effective_homotopy_colimits}

We continue the notation of Section~\ref{s:hocolim}.

\begin{proof}
We apply Lemma~\ref{l:filt} to the natural filtration $F_iC_*|\mcS| = C_*\sk_i|\mcS|$, where $\sk_i|\mcS|$ is the preimage of the $i$-skeleton $\sk_i\stdsimp 2$ under the natural projection $|\mcS| \to \stdsimp 2$. The Eilenberg--Zilber reduction applies to the quotient
\[C_*\sk_2|\mcS|/C_*\sk_1|\mcS| \cong C_*(\stdsimp 2 \times Z, \partial\stdsimp 2 \times Z) \Ra C_*(\stdsimp 2,\partial\stdsimp 2) \otimes C_*Z \cong s^2C_*Z\]
where $s$ denotes the suspension. The effective homology of $Z$ provides a further strong equivalence with $s^2C_*^\ef Z$. Similarly, $C_*\sk_1|\mcS|/C_*\sk_0|\mcS|$ is isomorphic to
\[C_*((\bdry_2\stdsimp 2,\partial\bdry_2\stdsimp 2) \times Z) \oplus C_*((\bdry_1\stdsimp 2,\partial\bdry_1\stdsimp 2) \times Z_0) \oplus C_*((\bdry_0\stdsimp 2,\partial\bdry_0\stdsimp 2) \times Z_1)\]
and thus strongly equivalent to $sC_*^\ef Z \oplus sC_*^\ef Z_0 \oplus sC_*^\ef Z_1$. Finally, $C_*\sk_0|\mcS|$ is strongly equivalent to $C_*^\ef Z_0 \oplus C_*^\ef Z_1 \oplus C_*^\ef Z_2$.

The subcomplexes corresponding to $\bdry_2|\mcS|$ are formed by some of the direct summands above and are thus preserved by all the involved strong equivalences. This finishes the verification of the assumptions of Lemma~\ref{l:filt}.
\end{proof}

The following lemma was used in the proof of Proposition~\ref{prop:weak_H_space_structure}.

\begin{lemma}\label{lem:calculations}
The two components $\pn\mathop{\add}$ and $\qn\mathop{\add}$ defined in the proof of Proposition~\ref{prop:weak_H_space_structure} determine a map $\mathop{\add}\col\Pnew\htimes_B\Pnew\ra\Pnew$ and this map is a weak \Hopf space structure.

The two components $\pn\mathop{\inv}$ and $\qn\mathop{\inv}$ defined in Section~\ref{sec:H_space_constr} determine a map $\mathop{\inv}\col\Pnew\ra\Pnew$ and this map is a right inverse for $\add$.
\end{lemma}

\begin{proof}
The compatibility for $\add$:
\begin{align*}
\delta\qn\mathop{\add} & =\delta\big(\mathop{\add_{\qn\onew}}(\qn\htimes\qn)+M(\pn\htimes\pn)\big)=\mathop{\add_{\kno}}(\delta\htimes\delta)(\qn\htimes\qn)+m(\pn\htimes\pn) \\
& =\mathop{\add_{\kno}}(\kn\htimes \kn)(\pn\htimes\pn)+\big(\kn\mathop{\add}-\mathop{\add_{\kno}}(\kn\htimes \kn)\big)(\pn\htimes\pn) \\
& =\kn\mathop{\add}(\pn\htimes\pn)=\kn\pn\mathop{\add}
\end{align*}
The weak \Hopf space condition $\mathop{\add}\vartheta=\hnabla$ on $\Pnew$ verified for its two components:
\begin{align*}
\pn\mathop{\add}\vartheta & =\mathop{\add}(\pn\htimes\pn)\vartheta=\mathop{\add}\vartheta(\pn\hvee\pn)=\hnabla(\pn\hvee\pn)=\pn\hnabla \\
\qn\mathop{\add}\vartheta & =\big(\mathop{\add_{\qn\onew}}(\qn\htimes\qn)+M(\pn\htimes\pn)\big)\vartheta=\mathop{\add_{\qn\onew}}\vartheta(\qn\hvee\qn)+\underbrace{M\vartheta}_0(\pn\hvee\pn) \\
& =\hnabla(\qn\hvee\qn)=\qn\hnabla
\end{align*}
The compatibility for $\inv$:
\begin{align*}
\delta(-c+2\qn\onew&-M(2,x,-x))=-\delta c+2\delta\qn\onew-m(2,x,-x) \\
& =-\kn x+2\kno-(\kn(\underbrace{x+(-x)}_{\oold})-\kn x+\kno-\kn(-x))=\kn(-x)
\end{align*}
The condition $\mathop{\add}(\id\htimes\mathop{\inv})\hDelta=\onew$ of being a right inverse:
\begin{align*}
(x,c)+(-(x,c)) & =(x,c)+(-x,-c+2\qn\onew-M(2,x,-x)) \\
& =(x+(-x),c+(-c+2\qn\onew-M(2,x,-x))-\qn\onew+M(2,x,-x)) \\
& =(\oold,\qn\onew)=\onew\qedhere
\end{align*}
\end{proof}

\ifpoly
\section{Polynomiality}\label{sec:polynomiality}

\subsection*{Basic notions}

The algorithm of Theorem~\ref{thm:main_theorem} was described for a single generalized lifting-extension problem. To prove that its running time is polynomial, we will have to deal with the class of all generalized lifting-extension problems and also certain related classes, e.g.\ the class of Moore-Postnikov stages of a given height. We will base our analysis on the notion of a locally polynomial-time simplicial set, described in \cite{polypost}. Here, we will call it a polynomial-time family of simplicial sets.

Since we assume $\theconn$ to be fixed and our algorithms only access information up to dimension $2\theconn+2$, we make the following standing assumption.

\begin{convention}\label{conv:encoding}
In this section, when speaking about the running time of algorithms, it is understood that inputs are limited to dimension at most $\theDim$ for some fixed $\theDim$.

Simplicial sets will be equipped with a choice of encoding of their simplices; thus, from now on, different choices of encoding of simplices of one simplicial set actually specify \emph{different} simplicial sets. The same applies to chain complexes etc.
\end{convention}

Usually, a collection $(X(p))_{p \in \sfP}$ of simplicial sets is understood as a mapping $p \mapsto X(p)$, associating to each $p \in \sfP$ a simplicial set $X(p)$. For technical reasons, our collections will also permit multi-valued mappings, i.e.\ $X(p)$ in general is not a single simplicial set but any of a number of simplicial sets. In effect, this is given by a relation between simplicial sets and parameters $p \in \sfP$, namely: $Z \sim p$ if and only if $Z$ is one of the possible values $X(p)$.

\begin{definition}\label{def:family}
A \emph{family of (locally effective) simplicial sets} is a collection $(X(p))_{p\in\sfP}$ of simplicial sets (equipped with choices of encodings of their simplices), as above, such that the elements of the \emph{parameter set} $\sfP$ have a representation in a computer and such that there are provided algorithms, all taking as inputs pairs $(p,x)$ with $p\in\sfP$ and $x \in X(p)$, and performing the following tasks:
\begin{enumerate}[labelindent=.5em,leftmargin=*,label=$\bullet$,itemsep=0pt,parsep=0pt,topsep=5pt]
\item
compute the $i$-th face of $x$,
\item
compute the $i$-th degeneracy of $x$,
\item
compute the action of $a\in G$ on $x$,
\item
compute the expression of $x$ as $x=ay$ with $a\in G$ and $y$ distinguished.
\end{enumerate}
We say that this family is \emph{polynomial-time} if all these algorithms have their running time bounded by $g(\size(p)+\size(x))$, where $g$ is some polynomial function and $\size(p)$, $\size(x)$ are the encoding sizes of $p$ and $x$ (we recall the assumption $\dim x\leq \theDim$).

A \emph{family of effective simplicial sets} possesses, in addition, an algorithm that, given $p\in\sfP$, outputs the list of all non-degenerate distinguished simplices of $X(p)$. (For such a family, the simplicial set $X(p)$ is necessarily unique.)
\end{definition}

Now we are able to explain the non-uniqueness of a represented simplicial set $X(p)$ for a given parameter $p \in \sfP$: namely, for any simplicial subset $A \subseteq X(p)$, we may use the same parameter $p$ to compute faces etc.\ in $A$, which means that $A$ might also be used as a value $X(p)$ at $p$.

\begin{example}
In Section~\ref{sec:equi_eff_hlgy_alg}, we described a way of encoding finite simplicial sets by listing all distinguished non-degenerate simplices and also the relations $d_j x = a s_I y$, for all $x$ distinguished non-degenerate and all possible $j$. Such encodings comprise the parameter set $\SSet$; it then supports an obvious polynomial-time family of effective simplicial sets, whose simplices are encoded as formal expressions $a s_I y$.
\end{example}

The notion of a family can be similarly defined for pairs of (effective) simplicial sets, (effective) chain complexes, strong equivalences, simplicial sets with effective homology etc. Each such class $\mcC$ is described by a collection of algorithms that are required to specify its object, similar to the list in Definition~\ref{def:family}. Families of objects of $\mcC$ are then the obvious parametrized versions of such collections of algorithms; we will denote them $C \colon \sfP\family\mcC$.

When constructing new families of objects, it is important that the resulting families are polynomial-time whenever the old ones are. We will encapsulate this situation in the notion of a polynomial-time construction. A \emph{construction} $F\col\mcC\to\mcD$ is simply a mapping; we use a different name to emphasize that it operates on the level of objects, i.e.\ mathematical structures of some sort, and their encodings (see Convention~\ref{conv:encoding}). In general, we will not require $F$ to be single-valued, having in mind an example of associating to an equation its solution -- no solution needs to exist and if it does exist, there might be many choices.

\begin{example}
There is an obvious construction
\[\xymatrix{
\{\text{couples of locally effective simplicial sets}\} \ar[r] & \{\text{locally effective simplicial sets}\}
}\]
(a couple is general; a pair is a couple $(X, A)$ with $A \subseteq X$). There is also an obvious way of transforming a couple of locally effective simplicial sets $X$, $Y$ into a locally effective simplicial set $X \times Y$, e.g.\ the $\thedimm$-th face $d_\thedimm(x,y)=(d_\thedimm x,d_\thedimm y)$ may be easily computed with the help of the corresponding algorithms for $X$ and $Y$.
\end{example}

We say that the construction $F$ is \emph{computable}, if there is given a collection of algorithms, which are allowed to use formal calls to algorithms describing a computable object $Z \in \mcC$ (i.e.\ a family of objects parametrized by a 1-element $\sfP$), that describe its image $F(Z)$ for arbitrary computable $Z \in \mcC$. We have seen an example above for the product construction -- the algorithm for $d_i$ in $X \times Y$ uses calls to algorithms for $d_i$ in $X$ and $Y$.

Given a family $C\colon\sfP\family\mcC$, we may replace the formal calls by calls to actual algorithms present in the family $C$ and thus obtain a family $\sfP \family \mcD$; we denote the resulting family by $F_*C \colon \sfP\family[C]\mcC\xra{F}\mcD$. A computable construction is said to be \emph{polynomial-time} if, in this way, one obtains a polynomial-time family $F_*C$ for every polynomial-time family $C$.

\begin{remark}
Since, for each class $\mcC$, the number of the required algorithms is finite, we may consider the parameter set $\Alg(\mcC)$, whose elements are such collections of algorithms (non-parametrized, i.e.\ describing a single computable object). Further, we denote by $\Alg_g(\mcC)$ the collections of algorithms that run in time bounded by the polynomial $g$. Then $\Alg(\mcC)$ supports an obvious family $\Alg(\mcC)\family\mcC$ that assigns to each collection of algorithms an object they represent (there may be many) and the parametrized version of each algorithm simply runs the appropriate algorithm contained in the parameter. It restricts to a polynomial-time family parametrized by $\Alg_g(\mcC)$.

With this notation, a \emph{computable} construction $F$ is a family structure on the collection $\Alg(\mcC)\family\mcC\to\mcD$. Moreover, $F$ is \emph{polynomial-time} if, in addition, this family restricts to a polynomial-time family $\Alg_g(\mcC)\family\mcD$ for each polynomial $g$.
\end{remark}

The dual situation is called a reparametrization: when $\Phi\col\sfQ\to\sfP$ is a polynomial-time mapping and $\sfP$ supports a polynomial-time family $C\col\sfP\family\mcC$ then $\Phi^*C \colon \sfQ\xra{\Phi}\sfP\family[C]\mcC$ is another polynomial-time family.

The main result of this section is the following.

\begin{theorem}\label{thm:main_poly}
For each fixed $\theconn\geq 1$, the algorithm of Theorem~\ref{thm:main_theorem} describes a polynomial-time construction
\setlength{\hlp}{\widthof{${}\cup\{\emptyset\},$}*\real{0.5}}
\[\xymatrix@R=15pt{
\left\{\parbox{\widthof{\upshape $\theconn$-stable generalized lifting-extension problems}}{\upshape $\theconn$-stable generalized lifting-extension problems composed of effective simplicial sets}\right\} \ar[r] & *+!!<-\the\hlp,\the\fontdimen22\textfont2>{\left\{\parbox{\widthof{\upshape abelian groups}}{\upshape fully effective abelian groups}\right\}\cup\{\emptyset\},} \\
{} \POS*!<0pt,-13pt>\xybox{\xymatrix@=15pt{
\scriptstyle A \ar[r] \ar@{ >->}[d] & \scriptstyle Y \ar[d] \\
\scriptstyle X \ar[r] & \scriptstyle B
}} \ar@{|->}[r] & [X,Y]^A_B,
}\]
where the $\theconn$-stability of a generalized lifting-extension problem means that $\dim X\leq 2\theconn$, both $B$ and $Y$ are simply connected and the homotopy fibre of $\psi\col Y\to B$ is $\theconn$-connected.
\end{theorem}

From the definition, we are required to set up a polynomial-time family indexed by $\Alg(\mcC)$ where $\mcC$ is the class of generalized lifting-extension problems in question. We will make use of restricted parameter sets $\Map$ and $\Pair$ that describe $\psi$ and $(X,A)$ respectively.

The whole computation is summarized in the following chains of computable functions between parameter sets that describe various partial stages of the computation; we will explain all the involved parameter sets later. The functions
\[\xymatrix@R=2pt{
\rightbox{\Map={}}{\EMPS0} \ar@{.>}[rr] \POS[];[rr]**\dir{}?<>(.33)**\dir{-}*\dir{>},[];[rr];**\dir{}?<>(.33)**\dir{-} & & \EMPS\thedim & \Map\times_\SSet\MPS\thedim \ar[r] & \MPS\thedim,
}\]
for $\thedim=\dim X$, describe the computation of the Moore--Postnikov system over $B$ and its pullback to $X$,
\[\xymatrix@R=2pt{
\Pair\times_\SSet\HMPS[\theotherdim-1]\thedim \ar[r] & \PMPS[\theotherdim]\thedim\cup\{\bot\} & \PMPS[\theotherdim]\thedim \ar[r] & \HMPS[\theotherdim]\thedim,
}\]
for $\theotherdim\leq\thedim$, describe the computation of the weak \Hopf space structure on the stable part of the pullback (when it admits a section at all) and
\[\xymatrix@R=2pt{
\Gamma_{\thedim}\col\Pair\times_\SSet\HMPS[\thedim]\thedim \ar@{~>}[r] & \{\text{fully effective abelian groups}\}
}\]
describes a polynomial-time family, given by the homotopy classes of sections of the final $\thedim$-th stage that are zero on $A$.

\subsection*{Moore--Postnikov systems}

The elements of the parameter set $\EMPS\thedim$ encode extended Moore--Postnikov systems and are composed of the following data
\begin{enumerate}[labelindent=.5em,leftmargin=*,label=$\bullet$,itemsep=0pt,parsep=0pt,topsep=5pt]
\item
finite simply connected simplicial sets $Y$, $B$;
\item
finitely generated abelian groups $\pi_1,\ldots,\pi_\thedim$;
\item
effective Postnikov invariants $\konepef,\ldots,\knpef$ (to be explained below);
\item
a simplicial map $\varphi_\thedim\col Y\to P_\thedim$;
\end{enumerate}
where we set, by induction, $P_0=B$ and $P_\thedimm=P_{\thedimm-1}\times_{K(\pi_\thedimm,\thedimm+1)}E(\pi_\thedimm,\thedimm)$, a pullback taken with respect to the Postnikov invariant $\kip\col P_{\thedimm-1}\to K(\pi_\thedimm,\thedimm+1)$ that corresponds to the equivariant cocycle
\[\xymatrix{
C_{\thedimm+1}P_{\thedimm-1} \ar[r] & C_{\thedimm+1}^\ef P_{\thedimm-1} \ar[r]^-{\kipef} & \pi_\thedimm
}\]
with the first map the obvious one coming from the effective homology of $P_{\thedimm-1}$. Thus, $\kipef$ is required to be an equivariant cocycle as indicated.\footnote{When $Y$ is not finite, $\varphi_\thedim$ has to be replaced by a certain collection of effective cochains on $Y$; details are explained in \cite{polypost}.}

The above simplicial sets provide a number of families
\[\xymatrix{
B,P_1,\ldots,P_\thedim,Y\col\EMPS\thedim \ar@{~>}[r] & \left\{\parbox{\widthof{simplicial sets with}}{simplicial sets with effective homology}\right\}
}\]
and also a number of families of simplicial maps $p_\thedimm$, $\kip$, $\varphi_\thedimm$ etc.\ between these. They are polynomial-time essentially by the results of \cite{polypost} -- there is only one significant difference, namely the (equivariant) polynomial-time homology of Moore--Postnikov stages. For that we need the following observation: the functor $B$ of Theorem~\ref{t:vokrinek} is a polynomial-time construction defined on
\[\left\{\parbox{\widthof{strong equivalences $C\LRa C^\ef$ with $C$ locally}}{strong equivalences $C\LRa C^\ef$ with $C$ locally effective over $\ZG$, $C^\ef$ effective over $\bbZ$}\right\}\]
and taking values in a similar class with everything $\ZG$-linear. Its polynomiality is guaranteed by the explicit nature of this functor, see~\cite{Vokrinek}.

Polynomiality of functions $\EMPS{\thedimm-1}\to\EMPS\thedimm$ is proved in the same way as in \cite{polypost} with the exception of the use of Proposition~\ref{prop:projectivity} that describes a polynomial-time construction
\[\xymatrix@R=2pt{
\left\{\parbox{\widthof{$n$-connected effective}}{$n$-connected effective chain complexes}\right\} \ar[r] & \left\{\parbox{\widthof{homomorphisms of effective}}{homomorphisms of effective abelian groups}\right\}, \\
C \ar@{|->}[r] & (C_{n+1}\to Z(C_{n+1})).
}\]

Parameters for a Moore--Postnikov system are comprised of the same data with the exception of $Y$ and $\varphi_\thedimm$; we denote their collection by $\MPS\thedim$. The parameters for the pullback $g^*S$ of a Moore--Postnikov system $S$ of $\psi\col Y\to B$ along $g\col \tB\to B$ are: the base is $\tB$, the homotopy groups remain the same and the Postnikov invariants are pulled back along $\tB\times_BP_\thedimm\to P_\thedimm$. Thus, the pullback function $\Map\times_\SSet\EMPS\thedim\to\MPS\thedim$, $(g,S)\mapsto g^*S$ is polynomial-time (it is defined whenever the target of $g$ agrees with the base of $S$).

\subsection*{Stable Moore--Postnikov systems}

For the subsequent developement, the most important ingredient is Lemma~\ref{l:lift_ext_one_stage}. It is easy to see that it is a polynomial-time construction
\[\xymatrix{
\left\{\parbox{\widthof{$(X,A)$ equipped with effective homology, $\pi$ fully}}{$(X,A)$ equipped with effective homology, $\pi$ fully effective abelian group, $z\col X\to K(\pi,\thedim+1)$, $c\col A\to E(\pi,\thedim)$ computable such that $\delta c=z|_A$}\right\} \ar[r] & \left\{\parbox{\widthof{$X\to E(\pi,\thedim)$}}{$X\to E(\pi,\thedim)$ computable}\right\}\cup\{\bot\}.
}\]
It will be useful to split this construction into two steps:  finding an ``effective'' cochain $c_0^\ef\col C_\thedim^\ef(X,A)\to\pi$ and computing from it the solution $\widetilde c+c_0$. The advantage of this splitting lies in the possibility of storing the effective cochain as a parameter.

We enhance the parameter set $\MPS\thedim$ to $\PMPS[\theotherdim]\thedim$ by including the parameter
\begin{enumerate}[labelindent=.5em,leftmargin=*,label=$\bullet$,itemsep=0pt,parsep=0pt,topsep=5pt]
\item
a simplicial map $\om\col B\to P_\theotherdim$;
\end{enumerate}
and to $\HMPS[\theotherdim]\thedim$ by including in addition the parameters
\begin{enumerate}[labelindent=.5em,leftmargin=*,label=$\bullet$,itemsep=0pt,parsep=0pt,topsep=5pt]
\item
equivariant effective cochains $M_\thedimm^\ef\col C_\thedimm^\ef(P_{\thedimm-1}\htimes_B P_{\thedimm-1},P_{\thedimm-1}\hvee_B P_{\thedimm-1})\to\pi_\thedimm$, $1\leq\thedimm\leq\theotherdim$;
\end{enumerate}
which give the zero section and the addition in the Moore--Postnikov stages; for the latter, we use the observation above.

There are polynomial-time functions
\[\xymatrix{
\PMPS[\theotherdim]\thedim \ar[r] & \HMPS[\theotherdim]\thedim
}\]
which compute inductively the equivariant cochains $M_\thedimm$, $1\leq\thedimm\leq\theotherdim$, using Lemma~\ref{l:lift_ext_one_stage}.

\subsection*{Computing diagonals}

We describe a number of polynomial-time families supported by $\HMPS{}$ and its relatives. We restrict our attention to the pullback Moore--Postnikov system $\widetilde S$ over $X$ whose stages will be denoted $\tPnew$. Proposition~\ref{prop:addition_on_homotopy_classes}, that uses the polynomial-times addition in the Moore--Postnikov system and a polynomial-time construction of Proposition~\ref{prop:homotopy_lifting}, gives a polynomial-time family
\[\xymatrix@R=3pt{
\Gamma_{\thedim,\theotherdim}\col\Pair\times_\SSet\HMPS[\theotherdim]\thedim \ar@{~>}[r] & \{\text{semi-effective abelian groups}\} \\
((X,A),\widetilde S) \ar@{|->}[r] & [X,\widetilde P_\theotherdim]^A_X
}\]
(defined whenever the bigger space $X$ of the pair $(X,A)$ agrees with the base of $\widetilde S$) which is then extended to a polynomial-time family of semi-effective exact sequences from Theorem~\ref{thm:exact_sequence_long}. We assume, by induction, that $\Gamma_{\thedim,\theotherdim-1}$ has been already promoted to a polynomial-time family of fully effective abelian groups. The ``five lemma'' for fully effective structures, Lemma~\ref{l:ses}, provides a polynomial-time construction
\[\xymatrix{
\left\{\parbox{\widthof{semi-effective exact sequences}}{semi-effective exact sequences $A\to B\to C\to D\to E$ with $A$, $B$, $D$, $E$ fully effective}\right\} \ar[r] & \{\text{fully effective abelian groups}\}
}\]
sending each exact sequence to its middle term $C$. Thus, $\Gamma_{\thedim,\theotherdim}$ is enhanced to a polynomial-time family of fully effective abelian groups.

\subsection*{Computing zero sections}

It remains to analyze the function
\[\xymatrix{
\Pair\times_\SSet\HMPS[\theotherdim-1]\thedim \ar[r] & \PMPS[\theotherdim]\thedim\cup\{\bot\}.
}\]
By Theorem~\ref{thm:exact_sequence_short}, we obtain a polynomial-time family of affine homorphisms
\[\xymatrix{
\kmst\col[X,\widetilde P_{\theotherdim-1}]^A_X \ar[r] & H^{\theotherdim+1}_G(X,A;\pi_\theotherdim)
}\]
between fully effective abelian groups, parametrized by $\Pair\times_\SSet\HMPS[\theotherdim-1]\thedim$. Since Lemma~\ref{l:preimage} describes a polynomial-time construction, we obtain a section $\ommo$ that lifts to $\widetilde P_\theotherdim$ in polynomial time; this lift $\om$ is also computed in polynomial time using Proposition~\ref{prop:lift_ext_one_stage}.
\fi

\subsection*{Acknowledgement}
We are grateful to Ji\v{r}\'{i} Matou\v{s}ek and Uli Wagner for many useful discussions, comments and suggestions that improved this paper a great deal. Moreover, this paper could hardly exist without our long-term collaboration, partly summarized in \cite{CKMSVW11}, \cite{polypost} and \cite{ext-hard}.
	
\bibliographystyle{plain}
\bibliography{Postnikov}

\vskip 20pt
\vfill
\vbox{\footnotesize%
\noindent\begin{minipage}[t]{0.45\textwidth}
{\scshape Martin \v{C}adek, Luk\'a\v{s} Vok\v{r}\'inek}\\
Department of Mathematics and Statistics,\\
Masaryk University,\\
Kotl\'a\v{r}sk\'a~2, 611~37~Brno,\\
Czech Republic
\end{minipage}
\hfill
\begin{minipage}[t]{0.45\textwidth}
{\scshape Marek Kr\v{c}\'al}\\
Institute of Science and Technology Austria,\\
Am Campus 1, 3400~Klosterneuburg,\\
Austria
\end{minipage}
}

\end{document}